\documentclass[12pt,leqno,a4paper]{article}

\usepackage[english]{babel}
\usepackage{amsmath}
\usepackage{amsfonts}
\usepackage{amsthm}
\usepackage{amssymb}

\usepackage{xspace}
\usepackage{euscript}
\usepackage{graphicx}
\usepackage{amscd}
\usepackage{tabularx}

\usepackage{enumerate}
 \usepackage{epsfig} 
 \usepackage{graphics} 
\theoremstyle{plain}
\newtheorem{theo}{Theorem}[section]
\newtheorem{lem}[theo]{Lemma}
\newtheorem{prop}[theo]{Proposition}
\theoremstyle{definition}
\newtheorem{definition}[theo]{Definition}
\theoremstyle{remark}
\newtheorem{rem}[theo]{Remark}
\numberwithin{equation}{section}

\newcommand{\R}{\mathbb{R}}

\newcommand{\divrg}{\textrm{div}\,}

\title{Estimating the area of extreme inclusions in Reissner-Mindlin plates
\thanks{Antonino Morassi is supported by PRIN 2015TTJN95 ``Identificazione
e diagnostica di sistemi strutturali complessi''. Edi Rosset is supported by
PRIN 201758MTR2 ``Direct and inverse problems for partial differential equations: theoretical
aspects and applications'', by Progetto GNAMPA 2019 ``Propriet\`a delle soluzioni di equazioni alle
derivate parziali e applicazioni ai problemi inversi" Istituto Nazionale di Alta Matematica (INdAM) and by FRA2016 ``Problemi inversi, dalla stabilit\`a alla
ricostruzione", Universit\`a degli Studi di Trieste. }}
\author{Antonino Morassi\thanks{Dipartimento Politecnico di Ingegneria e Architettura,
Universit\`a degli Studi di Udine, via Cotonificio 114, 33100
Udine, Italy. E-mail: \textsf{antonino.morassi@uniud.it}} \  and Edi
Rosset\thanks{Dipartimento di Matematica e Geoscienze,
Universit\`a degli Studi di Trieste, via Valerio 12/1, 34127
Trieste, Italy. E-mail: \textsf{rossedi@units.it}}}

\begin{document}

\maketitle

\centerline{\textit{Dedicated to Sergio in occasion of his
$65^{th}$ birthday}}

\begin{abstract}
\noindent We derive upper and lower estimates of the area of
unknown defects in the form of either cavities or rigid inclusions
in Mindlin-Reissner elastic plates in terms of the difference
$\delta W$ of the works exerted by boundary loads on the defected
and on the reference plate. It turns out that the upper estimates
depend linearly on $\delta W$, whereas the lower ones depend
quadratically on $\delta W$. These results continue a line of
research concerning size estimates of extreme inclusions in
electric conductors, elastic bodies and plates.

\medskip


\medskip

\end{abstract}

\centerline{}

\section{Introduction}
\label{Introduction}

In the present paper we continue a line of research concerning the
identification of unknown defects inside Reissner-Mindlin plates.
As is well known, the Reissner-Mindlin theory gives a more refined
model for elastic plates with respect to the Kirchhoff-Love theory
and, in particular, it allows for an accurate description of
moderately thick plates, having thickness $h$ of the order of one
tenth of the dimension of the middle plane $\Omega$.

Perhaps, the simplest approach in detecting defects consists in
estimating their \textit{size}. In \cite{l:mrv18} we derived
constructive upper and lower bounds of the area of elastic
inclusions (\textit{size estimates}) in terms of the difference
between the works exerted by given boundary loads in deforming the
plate without and with inclusion. Here, we obtain constructive
size estimates for \textit{extreme} inclusions in the form of
either cavities or rigid inclusions. The interested reader can
refer, among others, to the papers \cite{ARS00-PAMS},
\cite{l:bcoz}, \cite{DiCLW13-ASNSP}, \cite{DiCLVW13-SIMA},
\cite{Ik98-JIIPP}, \cite{KaKiMi12}, \cite{KaMi13},
\cite{KSS97-SIMA}, \cite{MiNg12} for results and application of
the size estimate approach to various physical contexts.

Let $\Omega \times \left [ -\frac{h}{2},  \frac{h}{2} \right ]$ be
the plate, with $\Omega$ a bounded domain in $\R^2$, and let
$\mathbb P$ the fourth-order bending tensor and $S$ the shearing
matrix of the reference plate (i.e., without defects). Let us
denote by $D \times \left [ -\frac{h}{2},  \frac{h}{2} \right ]$,
$D \subset \subset \Omega$, the unknown defect to be determined.
Our experiment consists in applying the same (self-equilibrated)
transverse force field $\overline{Q}$ and couple field
$\overline{M}$ at the boundary of the plate, in presence and in
absence of the inclusion.

When $D$ represents a cavity, the infinitesimal transverse
displacement $w_c$ and the infinitesimal rigid rotation
$\varphi_c$ (of transverse material fiber to the middle plane
$\Omega$) satisfy the following Neumann boundary value problem
\begin{center}
\( {\displaystyle \left\{
\begin{array}{lr}
     {\rm div}(S(\varphi_c+\nabla w_c))=0,
      & \mathrm{in}\ \Omega \setminus \overline {D},
        \vspace{0.25em}\\
      {\rm div}({\mathbb P} \nabla \varphi_c) -S(\varphi_c+\nabla w_c)=0,
      & \mathrm{in}\ \Omega \setminus \overline {D},
          \vspace{0.25em}\\
      S(\varphi_c+\nabla w_c)\cdot n=\overline{Q}
      & \mathrm{on}\ \partial \Omega,
        \vspace{0.25em}\\
      ({\mathbb P} \nabla \varphi_c)n=\overline{M}, & \mathrm{on}\ \partial \Omega,
          \vspace{0.25em}\\
         S(\varphi_c+\nabla w_c)\cdot n=0
      & \mathrm{on}\ \partial D,
        \vspace{0.25em}\\
                ({\mathbb P} \nabla \varphi_c)n=0, &\mathrm{on}\ \partial {D},
          \vspace{0.25em}\\
\end{array}
\right. } \) \vskip -10.3em
\begin{eqnarray}
& & \label{eq:1.intro-dir-pbm-cav-1}\\
& & \label{eq:1.intro-dir-pbm-cav-2}\\
& & \label{eq:1.intro-dir-pbm-cav-3}\\
& & \label{eq:1.intro-dir-pbm-cav-4}\\
& & \label{eq:1.intro-dir-pbm-cav-5}\\
& & \label{eq:1.intro-dir-pbm-cav-6}
\end{eqnarray}
\end{center}
where $n$ is the outer unit normal to $\Omega$ and $D$,
respectively.

In case $D$ is a rigid inclusion, the analogous statical
equilibrium problem becomes the following mixed boundary value
problem
\begin{center}
\( {\displaystyle \left\{
\begin{array}{lr}
     {\rm div}(S(\varphi_r +\nabla w_r ))=0,
      & \mathrm{in}\ \Omega \setminus \overline{D},
        \vspace{0.25em}\\
      {\rm div}({\mathbb P} \nabla \varphi_r ) -S(\varphi_r +\nabla w_r)=0,
      & \mathrm{in}\ \Omega \setminus \overline {D},
          \vspace{0.25em}\\
      S(\varphi_r +\nabla w_r )\cdot n=\overline{Q}
      & \mathrm{on}\ \partial \Omega,
        \vspace{0.25em}\\
      ({\mathbb P} \nabla \varphi_r )n=\overline{M}, & \mathrm{on}\ \partial \Omega,
          \vspace{0.25em}\\
      \varphi_r=b,
      & \mathrm{in}\ \overline{D},
        \vspace{0.25em}\\
      w_r=-b \cdot x + a, &\mathrm{in}\ \overline{D},
\end{array}
\right. } \) \vskip -10.3em
\begin{eqnarray}
& & \label{eq:1.intro-dir-pbm-rig-1}\\
& & \label{eq:1.intro-dir-pbm-rig-2}\\
& & \label{eq:1.intro-dir-pbm-rig-3}\\
& & \label{eq:1.intro-dir-pbm-rig-4}\\
& & \label{eq:1.intro-dir-pbm-rig-5}\\
& & \label{eq:1.intro-dir-pbm-rig-6}
\end{eqnarray}
\end{center}
where $\varphi_r$ and $w_r$ are continuous functions through
$\partial D$, $a$ is any constant and $b$ is any $2$-dimensional
vector.

When $D$ is empty, the equilibrium of the undefective plate is
modelled by
\begin{center}
\( {\displaystyle \left\{
\begin{array}{lr}
     {\rm div}(S(\varphi_0+\nabla w_0))=0,
      & \mathrm{in}\ \Omega,
        \vspace{0.25em}\\
      {\rm div}({\mathbb P} \nabla \varphi_0) -S(\varphi_0+\nabla w_0)=0,
      & \mathrm{in}\ \Omega,
          \vspace{0.25em}\\
      S(\varphi_0+\nabla w_0)\cdot n=\overline{Q}
      & \mathrm{on}\ \partial \Omega,
        \vspace{0.25em}\\
      ({\mathbb P} \nabla \varphi_0)n=\overline{M}, & \mathrm{on}\ \partial \Omega.
          \vspace{0.25em}\\
\end{array}
\right. } \) \vskip -7.5em
\begin{eqnarray}
& & \label{eq:1.intro-dir-pbm-1}\\
& & \label{eq:1.intro-dir-pbm-2}\\
& & \label{eq:1.intro-dir-pbm-3}\\
& & \label{eq:1.intro-dir-pbm-4}
\end{eqnarray}
\end{center}
Let us define the following boundary integrals, which express the
works produced by the given boundary loads $\overline{Q}$,
$\overline{M}$ when $D$ is a cavity, a rigid inclusion, or $D$ is
absent:
\begin{equation}
  \label{eq:intro-works}
  W_c=
  \int_{\partial\Omega}\overline{Q}w_c+\overline{M}\cdot\varphi_c,
  \quad
   W_r = \int_{\partial \Omega} \overline{Q}w_r + \overline{M} \cdot
    \varphi_r, \quad
  W_0=
  \int_{\partial\Omega}\overline{Q}w_0+\overline{M}\cdot\varphi_0.
\end{equation}
Our size estimates are formulated in terms of the
\textit{normalized work gap}
\begin{equation}
  \label{eq:intro-normalized-work}
  \frac{\delta W}{W_0},
\end{equation}
where
\begin{equation}
  \label{eq:intro-work-gap}
  \delta W= W_c -W_0, \qquad \delta W= W_0 -W_r,
\end{equation}
respectively.

Upper and lower estimates require different mathematical tools
and, also, present different dependence on the work gap.
Precisely, upper estimates have \textit{linear character}, but
require additional a priori assumptions on the material (isotropy)
and on the defect $D$, namely the following \textit{Fatness
Condition}:
\begin{equation}
  \label{eq:intro-fatness-condt}
  \hbox{area} \{
  x \in D \ | \ \hbox{dist}(x, \partial \Omega) > h_1 \} \geq
  \frac{1}{2} \ \hbox{area} (D),
\end{equation}
where $h_1$ is a given parameter.

The isotropy assumption ensures the unique continuation property
in the form of a quantitative three spheres inequality for
solutions to the reference problem
\eqref{eq:1.intro-dir-pbm-1}--\eqref{eq:1.intro-dir-pbm-4}, which
was obtained in \cite{l:mrv18}. Let us observe that the above
Fatness Condition could be removed provided a doubling inequality
were at disposal. In that case, the upper estimate would have
H\"{o}lder character, see, for example, \cite[Theorems $2.6$ and
$2.8$]{l:amr02} for an electric conductor, and \cite[Theorems
$2.7$ and $2.9$]{l:mr03} in the context of linear elasticity.

The estimates {}from below are quite different, both for the a
priori assumptions and the techniques of proof. We can replace the
Fatness Condition \eqref{eq:intro-fatness-condt} with the weaker
Scale Invariant Fatness Condition
\begin{equation}
  \label{eq:intro-SIFC}
  \hbox{diam} (D) \leq Q_D r_D,
\end{equation}
where $r_D$ is \textit{unknown}. This hypothesis avoids collapsing
of $D$ to an empty set having null $2$-dimensional Lebesgue
measure. On the other hand, we need to assume Lipchitz regularity of the boundary of $D$ and the dependence of $area(D)$ on $\delta
W$ has a quadratic character.

Main mathematical tools for cavities are constructive
Poincar\'{e}-type and Korn-type inequalities and, in particular, a
generalized Korn inequality recently obtained in \cite{l:mrv17},
suitable to handle the Reissner-Mindlin system.

The treatment of rigid inclusions requires, in addition, boundary
estimates in $L^2$ for the boundary value problem
\eqref{eq:1.intro-dir-pbm-rig-1}--\eqref{eq:1.intro-dir-pbm-rig-6}.
These estimates are based on identities of Rellich type (see
\cite{l:r}, \cite{l:pw}), and are in the style of the solvability
in $L^2$ of the regularity and Neumann problems formulated in
\cite{l:fkp}, \cite{l:kp}.

The paper is organized as follows. In Section \ref{sec:notation}
we introduce some useful notation. Direct problems are described
in Section \ref{sec:directproblems}, and the size estimates are
stated in Section \ref{sec:inverseproblems}. Finally, proofs are
given in Section \ref{SE-proof_cavities} and \ref{SE-proof_rigid},
for cavities and rigid inclusions, respectively.

\section{Notation} \label{sec:notation}

Let $P=(x_1(P), x_2(P))$ be a point of $\R^2$.
We shall denote by $B_r(P)$ the open disc in $\R^2$ of radius $r$ and
center $P$ and by $R_{a,b}(P)$ the rectangle
$R_{a,b}(P)=\{x=(x_1,x_2)\ |\ |x_1-x_1(P)|<a,\ |x_2-x_2(P)|<b \}$. To simplify the notation,
we shall denote $B_r=B_r(O)$, $R_{a,b}=R_{a,b}(O)$.

\begin{definition}
  \label{def:Lip_reg} (Lipschitz regularity)
Let $G$ be a bounded domain in ${\R}^{2}$. We
say that a portion $\Sigma$ of $\partial G$ is of
Lipschitz class with constants $\rho$, $L$
if, for any $P \in \Sigma$, there exists a rigid transformation of
coordinates under which we have $P=O$ and
\begin{equation*}
  G \cap R_{\rho,L\rho}=\{x=(x_1,x_2) \in R_{\rho,L\rho}\quad | \quad
x_{2}>\psi(x_1)
  \},
\end{equation*}
where $\psi$ is a Lipschitz continuous function on
$\left(-\rho,\rho\right)$ satisfying
\begin{equation*}
\psi(0)=0,
\end{equation*}
\begin{equation*}
\|\psi\|_{{C}^{0,1}\left(-\rho,\rho\right)} \leq L\rho.
\end{equation*}

\end{definition}
\begin{rem}
  \label{rem:2.1}
  We use the convention to normalize all norms in such a way that their
  terms are dimensionally homogeneous with the $L^\infty$ norm and coincide with the
  standard definition when the dimensional parameter equals one. For instance, the norm appearing above is meant as follows
\begin{equation}
  \label{eq:lip_norm}
  \|\psi\|_{{C}^{0,1}(-\rho,\rho)} =
  \|\psi\|_{{L}^{\infty}(-\rho,\rho)}+
  \rho\|\psi'\|_{{L}^{\infty}(-\rho,\rho)}.
\end{equation}
\end{rem}
Given $G\subset\R^2$, for any $t>0$ we denote
\begin{equation}
  \label{eq:int_env}
  G_{t}=\{x \in G \mid \hbox{dist}(x,\partial
  G)>t
  \},
\end{equation}
\begin{equation}
   \label{eq:ext_env}
   G^{r} = \{ x \in \R^{2}\quad | \quad 0<\textrm{dist}(x,G)<r \}.
\end{equation}

Let
\begin{equation}
  \label{eq:Affini}
    \begin{split}
 {\mathcal A}&=\{z=(\varphi,w)\ |\ \varphi = b, w=-b\cdot x+a, a\in \R, b\in \R^2
  \}\\
    &=\{z=(\varphi,w)\ |\ \nabla\varphi = 0, \varphi+\nabla w=0\}.
    \end{split}
\end{equation}

We denote by $\mathbb{M}^2$ the space of $2 \times 2$ real valued
matrices and by ${\mathcal L} (X, Y)$ the space of bounded linear
operators between Banach spaces $X$ and $Y$.

For every $2 \times 2$ matrices $A$, $B$ and for every $\mathbb{L}
\in{\mathcal L} ({\mathbb{M}}^{2}, {\mathbb{M}}^{2})$, we use the
following notation:
\begin{equation}
  \label{eq:notation_1}
  ({\mathbb{L}}A)_{ij} = L_{ijkl}A_{kl},
\end{equation}
\begin{equation}
  \label{eq:notation_2}
  A \cdot B = A_{ij}B_{ij}, \quad |A|= (A \cdot A)^{\frac {1}
  {2}}, \quad tr(A)=A_{ii},
\end{equation}
\begin{equation}
  \label{eq:notation_3}
  (A^T)_{ij}=A_{ji}, \quad \widehat{A} = \frac{1}{2}(A+A^T).
\end{equation}
Notice that here and in the sequel summation over repeated indexes
is implied.

\section{Formulation of the direct problems} \label{sec:directproblems}

Let us consider a plate, with constant thickness $h$, represented
by a bounded domain $\Omega$ in $\R^2$ having boundary of
Lipschitz class, with constants $\rho_0$ and $L_0$, and
satisfying
\begin{equation}
  \label{eq:Omegabound}
    \hbox{diam}(\Omega)\leq Q_0\rho_0,
\end{equation}
for some $Q_0 >0$, and
\begin{equation}
  \label{eq:ObelongsOmega}
    O\in \Omega.
\end{equation}

The \textit{reference} plate is assumed to be made by linearly
elastic isotropic material with Lam\'{e} moduli $\lambda$ and
$\mu$ satisfying the strong convexity conditions
\begin{equation}
  \label{eq:Lame-ell}
    \mu(x)\geq \alpha_0, \quad 2\mu(x)+3\lambda(x)\geq \gamma_0,
    \quad \hbox{in } \overline{\Omega},
\end{equation}
for given positive constants $\alpha_0$, $\gamma_0$, and the
regularity condition
\begin{equation}
  \label{eq:Lame-reg}
    \|\lambda\|_{C^{0,1}(\overline{\Omega})}+\|\mu\|_{C^{0,1}(\overline{\Omega})}\leq
    \alpha_1,
\end{equation}
where $\alpha_1$ is a given constant. Therefore, the shearing and
bending plate tensors take the form
\begin{equation}
  \label{eq:shearing-tensor}
    SI_2, \quad S=h\mu, \quad S \in
    C^{0,1}(\overline{\Omega}),
\end{equation}
\begin{equation}
  \label{eq:bending-tensor}
    {\mathbb P} A = B\left[(1-\nu)\widehat{A}+\nu tr(A)I_2\right],
    \quad {\mathbb P} \in
    C^{0,1}(\overline{\Omega}),
\end{equation}
where $I_2$ is the two-dimensional unit matrix, $A$ denotes a $2
\times 2$ matrix and
\begin{equation}
  \label{eq:Bending}
    B=\frac{Eh^3}{12(1-\nu^2)},
\end{equation}
with Young's modulus $E$ and Poisson's coefficient $\nu$ given by
\begin{equation}
  \label{eq:Young-Poisson}
   E=\frac{\mu(2\mu+3\lambda)}{\mu+\lambda}, \quad
   \nu=\frac{\lambda}{2(\mu+\lambda)}.
\end{equation}
By \eqref{eq:Lame-ell} and \eqref{eq:Lame-reg}, the ellipticity conditions for $S$ and $\mathbb P$ become
\begin{equation}
  \label{eq:convex-S-Lame}
    h \sigma_0 \leq S \leq h \sigma_1, \quad \hbox{in } \overline{\Omega},
\end{equation}
and
\begin{equation}
  \label{eq:convex-P-Lame}
     \frac{h^3}{12} \xi_0 | \widehat{A} |^2 \leq {\mathbb P}A \cdot A \leq \frac{h^3}{12} \xi_1 | \widehat{A} |^2, \quad \hbox{in } \overline{\Omega},
\end{equation}
for every $2\times 2$ matrix $A$, with
\begin{equation}
  \label{eq:constants_dependence}
     \sigma_0 =\alpha_0, \quad \sigma_1 =\alpha_1, \quad \xi_0=\min\{2\alpha_0, \gamma_0\}, \quad
        \xi_1=2\alpha_1.
\end{equation}
Moreover,
\begin{equation}
  \label{eq:lip_operatori}
   \|S\|_{ C^{0,1}(\overline{\Omega})} \leq h\alpha_1, \qquad
     \|{\mathbb P}\|_{C^{0,1}(\overline{\Omega})} \leq Ch^3,
\end{equation}
with $C>0$ only depending on $\alpha_0$, $\alpha_1$,
$\gamma_0$.

Let the boundary of the plate $\partial \Omega$ be subject to a
transverse force field $\overline{Q}$ and a couple field
$\overline{M}$ satisfying
\begin{equation}
  \label{eq:regolarita-carico}
    \overline{Q} \in H^{-\frac{1}{2}}(\partial \Omega), \quad \overline{M} \in H^{-\frac{1}{2}}(\partial \Omega,
    \R^2),
\end{equation}
and the compatibility conditions
\begin{equation}
  \label{eq:compatibilita-carico}
     \int_{\partial \Omega} \overline{Q}=0, \quad  \int_{\partial
     \Omega}(\overline{Q}x -\overline{M})=0.
\end{equation}
Throughout the paper, the defect is represented by an open set $D$
satisfying
\begin{equation}
  \label{eq:connessione}
     D \subset \subset \Omega, \qquad \Omega \setminus  \overline{D} \hbox{ is connected}.
\end{equation}
When $D$ represents a cavity, the statical equilibrium is governed
by the Neumann boundary value problem
\begin{center}
\( {\displaystyle \left\{
\begin{array}{lr}
     {\rm div}(S(\varphi_c+\nabla w_c))=0,
      & \mathrm{in}\ \Omega \setminus \overline {D},
        \vspace{0.25em}\\
      {\rm div}({\mathbb P} \nabla \varphi_c) -S(\varphi_c+\nabla w_c)=0,
      & \mathrm{in}\ \Omega \setminus \overline {D},
          \vspace{0.25em}\\
      S(\varphi_c+\nabla w_c)\cdot n=\overline{Q}
      & \mathrm{on}\ \partial \Omega,
        \vspace{0.25em}\\
      ({\mathbb P} \nabla \varphi_c)n=\overline{M}, & \mathrm{on}\ \partial \Omega,
          \vspace{0.25em}\\
         S(\varphi_c+\nabla w_c)\cdot n=0
      & \mathrm{on}\ \partial D,
        \vspace{0.25em}\\
                ({\mathbb P} \nabla \varphi_c)n=0, &\mathrm{on}\ \partial {D},
          \vspace{0.25em}\\
\end{array}
\right. } \) \vskip -10.3em
\begin{eqnarray}
& & \label{eq:1.dir-pbm-cav-1}\\
& & \label{eq:1.dir-pbm-cav-2}\\
& & \label{eq:1.dir-pbm-cav-3}\\
& & \label{eq:1.dir-pbm-cav-4}\\
& & \label{eq:1.dir-pbm-cav-5}\\
& & \label{eq:1.dir-pbm-cav-6}
\end{eqnarray}
\end{center}
where $n$ is the outer unit normal to $\Omega$ and to $D$,
respectively. A weak solution to the above system is a pair
$(\varphi_c,w_c)\in H^1(\Omega \setminus \overline{D}, \R^2)\times
H^1(\Omega \setminus \overline{D})$ satisfying
\begin{equation}
  \label{eq:weak_cav}
\int_{\Omega\setminus \overline{D}} {\mathbb P}\nabla \varphi_c\cdot \nabla \psi + \int_{\Omega\setminus \overline{D}} S(\varphi_c+\nabla w_c)\cdot
        (\psi+\nabla v)=\int_{\partial\Omega}\overline{Q} v + \overline{M}\cdot \psi,
\end{equation}
for every $\psi\in H^1(\Omega \setminus \overline{D}, \R^2)$ and for every $v\in H^1(\Omega \setminus \overline{D})$.

Problem \eqref{eq:1.dir-pbm-cav-1}--\eqref{eq:1.dir-pbm-cav-6} admits a weak solution $(\varphi_c,w_c)\in H^1(\Omega\setminus \overline{D}, \R^2)\times
H^1(\Omega\setminus \overline{D})$, which is uniquely determined up to addition of an element $z\in {\mathcal A}$.

Since $D$ has Lipschitz
boundary, one can continue a solution pair $(\varphi_c,w_c)$ to a $H^1(\Omega, \R^2)\times H^1(\Omega)$ function pair, which we continue to call $(\varphi_c,w_c)$:
\begin{equation}
  \label{eq:+-}
  (\varphi_c,w_c)=\left\{ \begin{array}{ll}
 (\varphi_c^{+},w_c^{+}) & \textrm{in } \Omega \setminus
  \overline{D},\\
  &  \\
(\varphi_c^{-},w_c^{-}) & \textrm{in }  D,
  \end{array}\right.
\end{equation}
where $(\varphi_c^{+}, w_c^{+})$ is the given solution $(\varphi_c,w_c)$ and $(\varphi_c^{-}, w_c^{-})$ is defined as the weak solution of the
Dirichlet problem

\begin{equation}
  \label{eq:Dir}
  \left\{ \begin{array}{ll}
   {\rm div}(S(\varphi_c^- +\nabla w_c^-))=0, &
  \mathrm{in}\  D,\\
    {\rm div}({\mathbb P} \nabla \varphi_c^-) -S(\varphi_c^-+\nabla w_c^-)=0,&
  \mathrm{in}\  D,\\
  \varphi_c^{-} = \varphi_c^{+} |_{\partial D}, & \mathrm{on}\ \partial D,\\
    w_c^{-} = w_c^{+} |_{\partial D}, & \mathrm{on}\ \partial D.
  \end{array}\right.
\end{equation}

When $D$ represents a rigid inclusion, the statical equilibrium is
governed by the mixed boundary value problem
\begin{center}
\( {\displaystyle \left\{
\begin{array}{lr}
     {\rm div}(S(\varphi_r^+ +\nabla w_r^+ ))=0,
      & \mathrm{in}\ \Omega \setminus \overline{D},
        \vspace{0.25em}\\
      {\rm div}({\mathbb P} \nabla \varphi_r^+ ) -S(\varphi_r^+ +\nabla w_r^+)=0,
      & \mathrm{in}\ \Omega \setminus \overline {D},
          \vspace{0.25em}\\
      S(\varphi_r^+ +\nabla w_r^+ )\cdot n=\overline{Q}
      & \mathrm{on}\ \partial \Omega,
        \vspace{0.25em}\\
      ({\mathbb P} \nabla \varphi_r^+ )n=\overline{M}, & \mathrm{on}\ \partial \Omega,
          \vspace{0.25em}\\
      \varphi_r^- + \nabla w_r^-=0,
      & \mathrm{in}\ \overline{D},
        \vspace{0.25em}\\
      \nabla \varphi_r^-= 0, &\mathrm{in}\ \overline{D},
        \vspace{0.25em}\\
      w_r^- = w_r^+,
      & \mathrm{on}\ \partial D,
        \vspace{0.25em}\\
      \varphi_r^- = \varphi_r^+,
      & \mathrm{on}\ \partial D,
\end{array}
\right. } \) \vskip -13.3em
\begin{eqnarray}
& & \label{eq:1.dir-pbm-rig-1}\\
& & \label{eq:1.dir-pbm-rig-2}\\
& & \label{eq:1.dir-pbm-rig-3}\\
& & \label{eq:1.dir-pbm-rig-4}\\
& & \label{eq:1.dir-pbm-rig-5}\\
& & \label{eq:1.dir-pbm-rig-6}\\
& & \label{eq:1.dir-pbm-rig-7}\\
& & \label{eq:1.dir-pbm-rig-8}
\end{eqnarray}
\end{center}
where we have denoted by $(\varphi_r^-, w_r^-)$ and $(\varphi_r^+,
w_r^+)$ the restriction of the solution $(\varphi_r, w_r)$ in
$\overline{D}$ and in $\Omega \setminus \overline{D}$,
respectively, and $n$ is the outer unit normal to $\Omega$. For future reference, we notice that the
compatibility conditions \eqref{eq:compatibilita-carico} together
with the above formulation
\eqref{eq:1.dir-pbm-rig-1}--\eqref{eq:1.dir-pbm-rig-8} imply
\begin{equation}
  \label{eq:rigid-1-10}
    \int_{\partial D} S(\varphi_r^+ + \nabla w_r^+)\cdot n=0,
\end{equation}
\begin{equation}
  \label{eq:rigid-1-11}
    \int_{\partial D}  ( (\mathbb{P} \nabla \varphi_r^+)n -    (S(\varphi_r^+ + \nabla w_r^+)\cdot n)x
    )=0.
\end{equation}
{}From the mechanical point of view, the last two conditions
state the force balance and the couple balance of the rigid
inclusion $D$, respectively.

Let us introduce
\begin{equation}
  \label{eq:spazio-H1-D}
    H_D^1(\Omega, \R^2) \times H_D^1(\Omega)
    = \{ (\varphi,w) \in H^1(\Omega, \R^2)\times H^1(\Omega) | \
        (\varphi,w)_{|_D}\in {\mathcal A}
    \}.
\end{equation}
A pair $(\varphi_r, w_r) \in H_D^1(\Omega, \R^2) \times
H_D^1(\Omega)$ is a weak solution to
\eqref{eq:1.dir-pbm-rig-1}--\eqref{eq:1.dir-pbm-rig-8} if for
every $(\psi, v) \in H_D^1(\Omega, \R^2) \times H_D^1(\Omega)$ we
have
\begin{equation}
  \label{eq:weak-solut-rigid}
    \int_\Omega \mathbb P \nabla \varphi_r \cdot \nabla \psi +
    \int_\Omega S(\varphi_r + \nabla w_r) \cdot (\psi + \nabla v)
    = \int_{\partial \Omega} \overline{Q}v + \overline{M} \cdot
    \psi.
\end{equation}
Since $H_D^1(\Omega, \R^2) \times H_D^1(\Omega) $ is a closed
linear subspace of $H^1(\Omega, \R^2) \times H^1(\Omega)$, by
standard variational methods it can be proven that
\eqref{eq:1.dir-pbm-rig-1}--\eqref{eq:1.dir-pbm-rig-8} has a weak solution $(\varphi_c,w_c)\in H^1(\Omega, \R^2)\times
H^1(\Omega)$, which is uniquely determined up to addition of an element $z\in {\mathcal A}$.

It is convenient to consider also the reference plate, in absence
of inclusions, whose statical equilibrium is governed by the
following Neumann boundary value problem
\begin{center}
\( {\displaystyle \left\{
\begin{array}{lr}
     {\rm div}(S(\varphi_0+\nabla w_0))=0,
      & \mathrm{in}\ \Omega,
        \vspace{0.25em}\\
      {\rm div}({\mathbb P} \nabla \varphi_0) -S(\varphi_0+\nabla w_0)=0,
      & \mathrm{in}\ \Omega,
          \vspace{0.25em}\\
      S(\varphi_0+\nabla w_0)\cdot n=\overline{Q}
      & \mathrm{on}\ \partial \Omega,
        \vspace{0.25em}\\
      ({\mathbb P} \nabla \varphi_0)n=\overline{M}, & \mathrm{on}\ \partial \Omega.
          \vspace{0.25em}\\
\end{array}
\right. } \) \vskip -7.5em
\begin{eqnarray}
& & \label{eq:1.dir-pbm-1}\\
& & \label{eq:1.dir-pbm-2}\\
& & \label{eq:1.dir-pbm-3}\\
& & \label{eq:1.dir-pbm-4}
\end{eqnarray}
\end{center}
A weak solution to the above Neumann problem is a pair $(\varphi_0,w_0)\in H^1(\Omega, \R^2)\times
H^1(\Omega)$ satisfying
\begin{equation}
  \label{eq:weak}
\int_{\Omega} {\mathbb P}\nabla \varphi_0\cdot \nabla \psi + \int_{\Omega} S(\varphi_0+\nabla w_0)\cdot
        (\psi+\nabla v)=\int_{\partial\Omega}\overline{Q} v + \overline{M}\cdot \psi,
\end{equation}
for every $\psi\in H^1(\Omega, \R^2)$ and for every $v\in H^1(\Omega)$.
The equilibrium problem \eqref{eq:1.dir-pbm-1}--\eqref{eq:1.dir-pbm-4} has a weak solution $(\varphi_0,w_0)\in H^1(\Omega, \R^2)\times
H^1(\Omega)$, which is uniquely determined up to addition of an element $z\in {\mathcal A}$.

Let us denote by $W_c$, $W_r$, $W_0$ the works exerted by the surface forces and couples $\overline{Q}$ and $\overline{M}$ when $D$ is a cavity, a rigid inclusion, or it is absent, respectively:
\begin{equation}
  \label{eq:W_c}
  W_c=
  \int_{\partial\Omega}\overline{Q}w_c+\overline{M}\cdot\varphi_c
    = \int_{\Omega\setminus \overline{D}}\mathbb{P}\nabla \varphi_c\cdot \nabla \varphi_c + \int_{\Omega\setminus \overline{D}}S(\varphi_c+\nabla w_c)\cdot(\varphi_c+\nabla w_c),
\end{equation}
\begin{equation}
  \label{eq:work-Wr}
    W_r = \int_{\partial \Omega} \overline{Q}w_r^+ + \overline{M} \cdot
    \varphi_r^+ =
    \int_{\Omega \setminus \overline{D}}\mathbb{P}\nabla \varphi_r^+\cdot \nabla \varphi_r^+ +
    \int_{\Omega\setminus \overline{D}}S(\varphi_r^+ +\nabla w_r^+)\cdot(\varphi_r^+ +\nabla
    w_r^+),
\end{equation}
\begin{equation}
  \label{eq:W_0}
  W_0=
  \int_{\partial\Omega}\overline{Q}w_0+\overline{M}\cdot\varphi_0
    = \int_{\Omega}\mathbb{P}\nabla \varphi_0\cdot \nabla \varphi_0 + \int_{\Omega}S(\varphi_0+\nabla w_0)\cdot(\varphi_0+\nabla w_0).
\end{equation}
\begin{rem}
\label{rem:works_independent} Let us notice that, in view of the
compatibility conditions \eqref{eq:compatibilita-carico},  the
works $W_c$, $W_r$, $W_0$ are well defined, that is they are
invariant with respect to the addition of any element $z$ in
$\mathcal A$ to the solution pair $(\varphi_c,w_c)$,
$(\varphi_r,w_r)$, $(\varphi_0,w_0)$, respectively.
\end{rem}

Throughout the paper, we shall choose the following normalization conditions for  $(\varphi_0, w_0)$:
\begin{equation}
  \label{eq:normalization-0}
 \int_\Omega \varphi_0 =0, \quad \int_\Omega w_0 =0.
\end{equation}

\section{The inverse problems: main results} \label{sec:inverseproblems}

Let us start considering the size estimates for cavities. We analyze separately the upper and lower estimates, since they require different a priori assumptions and techniques of proof.

\begin{theo}
  \label{theo:size_above_cavity}
  Let $\Omega$ be a bounded domain in $\R^{2}$, such that $\partial \Omega$
is of Lipschitz class with constants $\rho_{0}, L_{0}$ and
satisfying \eqref{eq:Omegabound}. Let $D$ be an open set
satisfying \eqref{eq:connessione} and
\begin{equation}
  \label{eq:D-fat_cavity}
  \left| D_{h_{1}\rho_0} \right| \geq \frac {1} {2} \left| D \right|,
  \end{equation}
for a given positive constant $h_{1}$. Let the reference plate be
made by linearly elastic isotropic material with Lam\'{e} moduli
$\lambda$, $\mu$ satisfying \eqref{eq:Lame-ell},
\eqref{eq:Lame-reg}. Let the transverse force field
$\overline{Q}\in H^{-1/2}(\partial \Omega)$ and the couple field $\overline{M}\in H^{-1/2}(\partial \Omega, \R^2)$ satisfy
\begin{equation}
  \label{eq:frequency}
    {\mathcal F} = \frac{\|\overline{M}\|_{H^{-1/2}(\partial \Omega)}+\rho_0
\|\overline{Q}\|_{H^{-1/2}(\partial \Omega)}}
{\|\overline{M}\|_{H^{-1}(\partial \Omega)}+\rho_0
\|\overline{Q}\|_{H^{-1}(\partial \Omega)}},
\end{equation}
for some positive constant ${\mathcal F}$.
The following estimate holds
\begin{equation}
  \label{eq:size_above_cavity}
  |D|
  \leq
  K\rho_0^2
  \frac{W_c-W_0}{W_0},
  \end{equation}
where $K$ only depends on $\alpha_{0}$, $\alpha_{1}$,
$\gamma_{0}$, $L_{0}$, $Q_0$,
$\frac{\rho_0}{h}$, $h_{1}$ and ${\mathcal F}$.
\end{theo}

In order to obtain the estimate {}from below, we assume that $D$ is a domain satisfying the following a priori assumptions concerning its regularity the shape:

\begin{equation}
  \label{eq:regul_D}
   \partial D \hbox{ is of Lipschitz class
with constants } r_D, L_D,
  \end{equation}
\begin{equation}
  \label{eq:SIFC}
  \hbox{diam}(D)\leq Q_D r_D,
  \end{equation}
where $L_D$, $Q_D$ are given a priori parameters, whereas $r_D$ is unknown.

Let us stress that $r_D$ is an \emph{unknown} parameter (otherwise, the size estimates should follow trivially), whereas the parameters $L_D$ and $Q_D$, which are invariant under scaling, will be considered as a priori information.

\begin{theo}
  \label{theo:size_below_cavity}
    Let $\Omega$ be a bounded domain in $\R^{2}$, such that $\partial \Omega$
is of Lipschitz class with constants $\rho_{0}, L_{0}$ and
satisfying \eqref{eq:Omegabound}. Let $D$ be a subdomain of
$\Omega$ satisfying \eqref{eq:connessione}, \eqref{eq:regul_D},
\eqref{eq:SIFC}, and such that
\begin{equation}
  \label{eq:D-away_boundary}
  \hbox{dist}(D,\partial\Omega)\geq d_0\rho_0,
\end{equation}
with $d_{0}>0$, $r_D<\frac{d_0}{2}\rho_0$.
Let the reference plate be
made by linearly elastic isotropic material with Lam\'{e} moduli
$\lambda$, $\mu$ satisfying \eqref{eq:Lame-ell},
\eqref{eq:Lame-reg}. Let the transverse force field
$\overline{Q}\in H^{-1/2}(\partial \Omega)$ and the couple field $\overline{M}\in H^{-1/2}(\partial \Omega, \R^2)$.
The following estimate holds
\begin{equation}
  \label{eq:size_below_cavity}
  |D|
  \geq
  k\rho_0^2
  \Psi \left( \frac {W_{c}-W_{0}} {W_{0}} \right),
  \end{equation}
where the function $\Psi$ is given by
\begin{equation}
  \label{eq:cav_Psi}
  [0,+\infty)\ni t\mapsto\Psi (t) = \frac{t^{2}}{1+t},
\end{equation}
and $k>0$ only depends on $\alpha_{0}$, $\alpha_{1}$,
$\gamma_{0}$, $L_{0}$, $Q_0$,
$\frac{\rho_0}{h}$, $d_{0}$,
$L_D$, $Q_D$.
\end{theo}

Concerning rigid inclusions, the size estimates are stated in the next two theorems.

\begin{theo}
  \label{theo:size-above-rigid}
  Let $\Omega$ be a bounded domain in $\R^{2}$, such that $\partial \Omega$
is of Lipschitz class with constants $\rho_{0}, L_{0}$ and
satisfying \eqref{eq:Omegabound}. Let $D$ be an open set
satisfying \eqref{eq:connessione} and the \textit{fatness
condition} \eqref{eq:D-fat_cavity}. Let the reference plate be
made by linearly elastic isotropic material with Lam\'{e} moduli
$\lambda$, $\mu$ satisfying \eqref{eq:Lame-ell},
\eqref{eq:Lame-reg}. Let the transverse force field
$\overline{Q}\in H^{-1/2}(\partial \Omega)$ and the couple field
$\overline{M}\in H^{-1/2}(\partial \Omega, \R^2)$ satisfy
\eqref{eq:frequency}.

The following estimate holds
\begin{equation}
  \label{eq:size-above-rigid}
  |D|
  \leq
  K\rho_0^2
  \frac{W_0-W_r}{W_0},
  \end{equation}
where $K$ only depends on $\alpha_{0}$, $\alpha_{1}$,
$\gamma_{0}$, $L_{0}$, $Q_0$, $\frac{\rho_0}{h}$, $h_{1}$ and
${\mathcal F}$.
\end{theo}

\begin{theo}
  \label{theo:size-below-rigid}
Let $\Omega$ be a bounded domain in $\R^{2}$, such that $\partial
\Omega$ is of Lipschitz class with constants $\rho_{0}, L_{0}$ and
satisfying \eqref{eq:Omegabound}. Let $D$ be a subdomain of
$\Omega$ satisfying \eqref{eq:connessione}, \eqref{eq:regul_D},
\eqref{eq:SIFC}, and such that
\begin{equation}
  \label{eq:rigid-away-boundary}
  \hbox{dist}(D,\partial\Omega)\geq d_0\rho_0,
\end{equation}
with $d_{0}>0$, $r_D<\frac{d_0}{2}\rho_0$. Let the reference plate
be made by linearly elastic isotropic material with Lam\'{e}
moduli $\lambda$, $\mu$ satisfying \eqref{eq:Lame-ell},
\eqref{eq:Lame-reg}. Let the transverse force field
$\overline{Q}\in H^{-1/2}(\partial \Omega)$ and the couple field
$\overline{M}\in H^{-1/2}(\partial \Omega, \R^2)$. The following
estimate holds
\begin{equation}
  \label{eq:size-below-rigid}
  |D|
  \geq
  C\rho_0^2
  \Phi \left( \frac {W_{0}-W_{r}} {W_{0}} \right),
  \end{equation}
where the function $\Phi$ is given by
\begin{equation}
  \label{eq:rigid-Phi}
  [0,1)\ni t \mapsto \Phi(t) = \frac{t^{2}}{1-t},
\end{equation}
and $C>0$ only depends on $\alpha_{0}$, $\alpha_{1}$,
$\gamma_{0}$, $L_{0}$, $Q_0$, $\frac{\rho_0}{h}$, $d_{0}$, $L_D$,
$Q_D$.
\end{theo}

\begin{rem}
\label{rem:non_serve_isotropia} Let us notice that the estimates
{}from below stated in Theorems \ref{theo:size_below_cavity} and
\ref{theo:size-below-rigid} hold for the general context of
anisotropic plates with bounded coefficients satisfying the strong
convexity assumption, since unique continuation estimates are not
needed for the proofs of these theorems.

Moreover, Theorems \ref{theo:size_below_cavity} and
\ref{theo:size-below-rigid} can be extended to the case when $D$
is made of a finite \emph{unknown} number of connected components.
Precisely, it suffices to assume that $D=\cup_{j=1}^J D_j$,
$j=1,...,J$, with $\Omega\setminus\overline{D}$ connected,
$\partial D_j$ of Lipschitz class with constant $r_j$, $L_D$, such
that $\hbox{diam}(D_j)\leq Q_D r_j$, $\hbox{dist}(D_i,D_j)\geq
\frac{3}{2}(r_i+r_j)$, $r_j\leq \frac{d_0}{2}\rho_0$. The proofs
can be extended to this general case by applying the same
arguments to each connected component $D_j$ taking care to replace
the integrals over $\Omega\setminus \overline{D}$ with integrals
over a neighborhood of $\partial D_j$ in $\Omega\setminus
\overline{D}$, by summing up the estimates obtained for each $j$,
and by applying the Cauchy-Schwarz inequality.
\end{rem}

\section{Proof of the size estimates for cavities}
\label{SE-proof_cavities}

The proofs of both Theorems \ref{theo:size_above_cavity}, \ref{theo:size_below_cavity} are based on the following energy lemma.

\begin{lem}
   \label{lem:double_inequality}
    Let $\Omega$ be a bounded domain in $\R^2$ and $D \subset \subset \Omega$ a measurable set. Let $S$, $\mathbb P$ given in
\eqref{eq:shearing-tensor}, \eqref{eq:bending-tensor} satisfy the
strong convexity conditions \eqref{eq:Lame-ell}. Let $(\varphi_c, w_c) \in
H^1(\Omega\setminus \overline{D},\R^2) \times H^1(\Omega\setminus \overline{D})$,
$(\varphi_0, w_0) \in H^1(\Omega,
\R^2) \times H^1(\Omega)$ be the weak solutions to problems \eqref{eq:1.dir-pbm-cav-1}--\eqref{eq:1.dir-pbm-cav-6} and \eqref{eq:1.dir-pbm-1}--\eqref{eq:1.dir-pbm-4}, respectively. We
have

\begin{multline}
  \label{eq:double_inequality}
\int_{D} {\mathbb P}\nabla \varphi_0\cdot \nabla \varphi_0 + \int_{D} S(\varphi_0+\nabla w_0)\cdot
        (\varphi_0+\nabla w_0)
    \leq W_c-W_0=\\
    =\int_{D} {\mathbb P}\nabla \varphi_c\cdot \nabla \varphi_0 + \int_{D} S(\varphi_c+\nabla w_c)\cdot
        (\varphi_0+\nabla w_0).
\end{multline}
    \end{lem}
\begin{proof}
For every weak solution $(\varphi,w) \in H^1(\Omega \setminus \overline{D}, \R^2)\times
H^1(\Omega\setminus \overline{D})$ to the system \eqref{eq:1.dir-pbm-cav-1}--\eqref{eq:1.dir-pbm-cav-2}, we have that
\begin{multline}
  \label{eq:double_ineq1}
  \int_{\Omega\setminus \overline{D}} {\mathbb P}\nabla \varphi\cdot \nabla \psi + \int_{\Omega\setminus \overline{D}} S(\varphi+\nabla w)\cdot
        (\psi+\nabla v)=\\
  =\int_{\partial\Omega}(S(\varphi + \nabla w)\cdot n)v
    +({\mathbb P} \nabla \varphi)n\cdot \psi
    -\int_{\partial D}(S(\varphi + \nabla w)\cdot n)v
    +({\mathbb P} \nabla \varphi)n\cdot \psi,
\end{multline}
for every $\psi\in H^1(\Omega \setminus \overline{D}, \R^2)$ and
for every $v\in H^1(\Omega \setminus \overline{D})$, where $n$
denotes the exterior unit normal to $\Omega$ and $D$,
respectively.

Choosing in the above identity $(\varphi,w)=(\varphi_0,w_0)$, $(\psi,v)=(\varphi_c,w_c)$, we have
\begin{multline}
  \label{eq:double_ineq2}
\int_{\Omega\setminus \overline{D}} {\mathbb P}\nabla \varphi_0\cdot \nabla \varphi_c + \int_{\Omega\setminus \overline{D}} S(\varphi_0+\nabla w_0)\cdot
        (\varphi_c+\nabla w_c)=\\
  =W_c-\int_{\partial D}(S(\varphi_0 + \nabla w_0)\cdot n)w_c
    +({\mathbb P} \nabla \varphi_0)n\cdot \varphi_c.
\end{multline}
Similarly, choosing in \eqref{eq:double_ineq1} $(\varphi,w)=(\varphi_c,w_c)$, $(\psi,v)=(\varphi_0,w_0)$ and recalling the boundary conditions \eqref{eq:1.dir-pbm-cav-5}--\eqref{eq:1.dir-pbm-cav-6}, we have
\begin{equation}
  \label{eq:double_ineq3}
\int_{\Omega\setminus \overline{D}} {\mathbb P}\nabla \varphi_c\cdot \nabla \varphi_0 + \int_{\Omega\setminus \overline{D}} S(\varphi_c+\nabla w_c)\cdot (\varphi_0+\nabla w_0)
        =W_0.
\end{equation}
By subtracting \eqref{eq:double_ineq3} {}from \eqref{eq:double_ineq2},
\begin{equation}
  \label{eq:double_ineq4}
W_c-W_0=\int_{\partial D}(S(\varphi_0 + \nabla w_0)\cdot n)w_c
    +({\mathbb P} \nabla \varphi_0)n\cdot \varphi_c.
\end{equation}
On the other hand, by the weak formulation of the system \eqref{eq:1.dir-pbm-1}--\eqref{eq:1.dir-pbm-2} in $D$, recalling the transmission conditions in \eqref{eq:Dir} for $(\varphi_c,w_c)$ and by \eqref{eq:double_ineq4}, it follows that
\begin{equation}
  \label{eq:double_ineq5}
\int_{D} {\mathbb P}\nabla \varphi_0\cdot \nabla \varphi_c + \int_{D} S(\varphi_0+\nabla w_0)\cdot
        (\varphi_c+\nabla w_c)=W_c-W_0,
    \end{equation}
that is the equality in \eqref{eq:double_inequality} is established.

Choosing in \eqref{eq:double_ineq1} $(\varphi,w)=(\psi,v)=(\varphi_c-\varphi_0,w_c-w_0)$, recalling that $(\varphi_c,w_c)$ and $(\varphi_0,w_0)$ satisfy the same Neumann conditions on $\partial\Omega$ and that $(\varphi_c,w_c)$ satisfies homogeneous Neumann conditions on $\partial D$, we have
\begin{multline}
  \label{eq:double_ineq6}
\int_{\Omega \setminus \overline{D}}{\mathbb P}\nabla (\varphi_c-\varphi_0)\cdot \nabla (\varphi_c-\varphi_0) + \\
+\int_{\Omega\setminus \overline{D}} S((\varphi_c-\varphi_0)+\nabla (w_c-w_0))\cdot ((\varphi_c-\varphi_0)+\nabla (w_c-w_0))=\\
=\int_{\partial D}(S(\varphi_0 + \nabla w_0)\cdot n)(w_c-w_0)
    +({\mathbb P} \nabla \varphi_0)n\cdot (\varphi_c-\varphi_0).
\end{multline}
Summing \eqref{eq:double_ineq3} and \eqref{eq:double_ineq5}, we obtain
\begin{equation}
  \label{eq:double_ineq7}
\int_{\Omega} {\mathbb P}\nabla \varphi_0\cdot \nabla \varphi_c + \int_{\Omega} S(\varphi_0+\nabla w_0)\cdot
        (\varphi_c+\nabla w_c)=W_c.
    \end{equation}
Subtracting \eqref{eq:W_0} {}from \eqref{eq:double_ineq7} and recalling \eqref{eq:double_ineq5}, we have
\begin{multline}
  \label{eq:double_ineq8}
\int_{\Omega} {\mathbb P}\nabla \varphi_0\cdot \nabla (\varphi_c-\varphi_0) + \int_{\Omega} S(\varphi_0+\nabla w_0)\cdot
        ((\varphi_c-\varphi_0)+\nabla (w_c-w_0))=\\
                =W_c-W_0=\int_{D} {\mathbb P}\nabla \varphi_0\cdot \nabla \varphi_c + \int_{D} S(\varphi_0+\nabla w_0)\cdot
        (\varphi_c+\nabla w_c).
    \end{multline}
By splitting the domain of integration on the left hand side of \eqref{eq:double_ineq8} into the union of $\Omega\setminus \overline{D}$ and $D$, the following identity easily follows
\begin{multline}
  \label{eq:double_ineq9}
\int_{\Omega\setminus \overline{D}} {\mathbb P}\nabla (\varphi_c-\varphi_0)\cdot \nabla \varphi_0 + \int_{\Omega\setminus \overline{D}} S((\varphi_c-\varphi_0)+\nabla (w_c-w_0))\cdot (\varphi_0+\nabla w_0)
        =\\
                =\int_{D} {\mathbb P}\nabla \varphi_0\cdot \nabla \varphi_0 + \int_{D} S(\varphi_0+\nabla w_0)\cdot
        (\varphi_0+\nabla w_0).
    \end{multline}

By adding and subtracting to the left hand side of \eqref{eq:double_ineq9} the term $\int_{\Omega\setminus \overline{D}} {\mathbb P}\nabla \varphi_c\cdot \nabla (\varphi_c-\varphi_0) + \int_{\Omega\setminus \overline{D}} S(\varphi_c+\nabla w_c)\cdot
        ((\varphi_c-\varphi_0)+\nabla (w_c-w_0))$ and recalling \eqref{eq:W_c} and \eqref{eq:double_ineq3}, we derive
\begin{multline}
  \label{eq:double_ineq11}
  \int_{D} {\mathbb P}\nabla \varphi_0\cdot \nabla \varphi_0 + \int_{D} S(\varphi_0+\nabla w_0)\cdot
        (\varphi_0+\nabla w_0)=\\
                =W_c-W_0-
    \int_{\Omega\setminus \overline{D}} {\mathbb P}\nabla (\varphi_c-\varphi_0)\cdot \nabla (\varphi_c-\varphi_0) - \\
    -\int_{\Omega\setminus \overline{D}} S((\varphi_c-\varphi_0)+\nabla (w_c-w_0))\cdot
        ((\varphi_c-\varphi_0)+\nabla (w_c-w_0))
    \leq W_c-W_0,
\end{multline}
that is the inequality in \eqref{eq:double_inequality}.
\end{proof}

It is convenient to introduce the strain energy
density
\begin{equation}
  \label{eq:energy}
   E(\varphi_0,w_0) = \left(|\widehat{\nabla} \varphi_0|^2 +\frac{1}{\rho^2_0} |\varphi_0 + \nabla
  w_0|^2\right)^\frac{1}{2}.
\end{equation}
Let us notice that, by \eqref{eq:convex-S-Lame}--\eqref{eq:constants_dependence},
the following double inequality holds
\begin{equation}
  \label{eq:E_double}
     m\rho_0^3E^2(\varphi_0, w_0)\leq {\mathbb P}\nabla \varphi_0\cdot \nabla \varphi_0
        + S(\varphi_0+\nabla w_0)\cdot
        (\varphi_0+\nabla w_0)\leq M \rho_0^3E^2(\varphi_0, w_0),
\end{equation}
where $m=\min\left\{\frac{\xi_0}{12}\left(\frac{h}{\rho_0}\right)^3,\alpha_0\frac{h}{\rho_0}\right\}$,
$M=\max\left\{\frac{\xi_1}{12}\left(\frac{h}{\rho_0}\right)^3,\sigma_1\frac{h}{\rho_0}\right\}$ only depend on $\alpha_0$, $\alpha_1$, $\gamma_0$ and $\frac{\rho_0}{h}$.

The second key tool for proving Theorem \ref{theo:size_above_cavity} is the following unique continuation property for solutions to \eqref{eq:1.dir-pbm-1}--\eqref{eq:1.dir-pbm-4}.

\begin{prop} [Lipschitz propagation of smallness]
   \label{theo:LPS}
Under the assumptions of Theorem \ref{theo:size_above_cavity},
for every $\rho>0$ and for every
$x\in\Omega_{\frac{7}{2\theta}\rho}$, we have
\begin{equation}
  \label{eq:LPS}
     \int_{B_{\rho}(x)} E^2(\varphi_0, w_0)\geq C_\rho\int_{\Omega} E^2(\varphi_0, w_0),
\end{equation}
where $C_\rho$ only depends on $\alpha_0$, $\alpha_1$, $\gamma_0$,
$\frac{\rho_0}{h}$, $L_0$, $Q_0$, ${\mathcal F}$,
$\frac{\rho}{\rho_0}$, and where $\theta\in (0,1)$ only depends on
$\alpha_0,\alpha_1,\gamma_0, \frac{\rho_0}{h}$.
\end{prop}

The above proposition was established in \cite[Theorem 4.5]{l:mrv18}.

\begin{proof}[Proof of Theorem \ref{theo:size_above_cavity}]
Let us cover $D_{h_{1}\rho_0}$ with internally non overlapping closed
squares $Q_{j}$ of side $l$, for $j=1,...,J$, with $l=\frac
{4\theta h_1}{2\sqrt 2\theta +7}\rho_0$, where $\theta\in (0,1)$ is as
in Proposition \ref{theo:LPS}. By the choice of $l$ the squares
$Q_{j}$ are contained in $D$. Hence
\begin{equation}
  \label{eq:proof_cavity_above1}
  \int_{D} E^2(\varphi_0,w_0)
  \geq
  \int_{\bigcup_{j=1}^{J} Q_{j}} E^2(\varphi_0,w_0)
  \geq
  \frac {|D_{h_{1}\rho_0}|}{l ^{2}}
  \int_{Q_{\bar {j}}} E^2(\varphi_0,w_0),
\end{equation}
where $\bar {j}$ is such that $\int_{Q_{\bar {j}}}
E^2(\varphi_0,w_0)=\min_{j}\int_{Q_{j}} E^2(\varphi_0,w_0)$. Let
$\bar {x}$ be the center of $Q_{\bar {j}}$. By applying estimate
\eqref{eq:LPS} with $x=\bar {x}$ and $\rho=l/2$ and using
\eqref{eq:proof_cavity_above1}  and \eqref{eq:D-fat_cavity}
we have
\begin{equation}
  \label{eq:proof_cavity_above2}
  \int_{D}E^2(\varphi_0,w_0)
  \geq
  C \frac{|D|}{\rho_0^2} \int_{\Omega}E^2(\varphi_0,w_0),
\end{equation}
where $C$ only depends on $\alpha_{0}$, $\alpha_{1}$,
$\gamma_{0}$, $L_{0}$, $Q_0$, $\frac{\rho_0}{h}$, $h_{1}$
and ${\mathcal F}$.

{}From the left hand side of \eqref{eq:double_inequality}, \eqref{eq:E_double}, \eqref{eq:proof_cavity_above2} and \eqref{eq:W_0}, we have
\begin{equation}
  \label{eq:proof_cavity_above3}
  W_c-W_0\geq m\rho_0^3\int_{D}E^2(\varphi_0,w_0)
  \geq
  C \rho_0|D|\int_{\Omega}E^2(\varphi_0,w_0)\geq C \frac{|D|}{\rho_0^2}W_0,
\end{equation}
with $C$ only depending on $\alpha_{0}$, $\alpha_{1}$,
$\gamma_{0}$, $L_{0}$, $Q_0$, $\frac{\rho_0}{h}$, $h_{1}$
and ${\mathcal F}$, so that \eqref{eq:size_above_cavity} follows.
\end{proof}

Let us premise some auxiliary propositions concerning Poincar\'{e} and Korn inequalities, which will be used for the proof of Theorem \ref{theo:size_below_cavity}. In the following three propositions $G$ is meant to be a bounded measurable domain in $\R^2$ having boundary of Lipschitz  class with constants $\rho$ and $L$ and satisfying
\begin{equation}
  \label{eq:diam_G}
  \hbox{diam}(G)\leq Q\rho.
  \end{equation}

Given $u \in H^{1}(G)$ and given $\Gamma \subset \partial G$, we shall denote
\begin{equation}
   \label{eq:medie}
   u_{G}=\frac{1}{| G |} \int_{G} u,
   \quad \quad
   u_{\Gamma}=\frac{1}{| \Gamma |} \int_{\Gamma}
   u.
\end{equation}
\medskip
\begin{prop} [Poincar\'{e}-type inequalities]
  \label{prop:Poinc}
For every $u \in H^{1}(G)$
  we have
\begin{equation}
   \label{eq:PoincG-a}
   \int_{G}|u-u_{G}|^{2} \leq C_{1}\rho^{2} \int_{G} | \nabla u
   |^{2},
\end{equation}
\begin{equation}
   \label{eq:PoincG-a-bis}
   \int_{G}|u-u_{\Gamma}|^{2} \leq C_{2} \left ( 1 + \frac{|G|}{\rho |\Gamma|}   \right ) \rho^2 \int_{G} | \nabla u
   |^{2},
\end{equation}
where $C_{1}$ and $C_2$ only depend on $L$, $Q$.

Moreover, if $u \in H^{1}(G^{\rho})$ then
\begin{equation}
   \label{eq:PoincGr}
   \int_{\partial G}|u-u_{\partial G}|^{2} \leq C_{3}\rho \int_{G^{\rho}} | \nabla u
   |^{2},
\end{equation}
where $C_{3}>0$ only depends on $L$, $Q$.
\end{prop}
The above Poincar\'{e}-type inequalities are well-known. A precise
evaluation of the constants $C_{1}$, $C_{2}$, $C_3$ in terms of
the scale invariant parameters $L$, $Q$ regarding the regularity and
the shape of $G$, can be found in the proof of \cite[Proposition
3.2]{l:amr02}.
\begin{prop} [Second Korn's inequality]
  \label{prop:Korn2}
  For every $\varphi \in H^{1}(G, \R^{2})$
  satisfying
\begin{equation}
   \label{eq:compkorn}
   \int_{G} ( \nabla \varphi - (\nabla \varphi)^T)=0,
\end{equation}
we have
\begin{equation}
   \label{eq:korn2}
   \int_{G} |\nabla \varphi|^2 \leq C \int_{G} |\widehat{\nabla} \varphi|^2
   ,
\end{equation}
where $C>0$ only depends on $L$ and $Q$.
\end{prop}
For a proof of this classical inequality see, for instance,
\cite{l:friedrichs}, \cite{l:necas}.

\begin{prop} [Generalized second Korn inequality]
   \label{prop:Korn_gener}
For every $\varphi\in H^1(G,\R^2)$ and for every $w\in H^1(G,\R)$,
\begin{equation}
  \label{eq:korn_gener}
    \|\nabla \varphi\|_{L^2(G)}\leq C\left(\|\widehat{\nabla} \varphi\|_{L^2(G)}+\frac{1}{\rho}\|\varphi+\nabla w\|_{L^2(G)}\right),
\end{equation}
where $C$ only depends on $L$ and $Q$.
\end{prop}

The above Generalized Korn inequality, established in
\cite[Theorem 4.3]{l:mrv17}, turned out to be a key result in
dealing with the direct Neumann problem for Reissner-Mindlin
plates (see Proposition 5.2 in \cite{l:mrv17}) and in deriving
unique continuation estimates for system
\eqref{eq:1.dir-pbm-1}--\eqref{eq:1.dir-pbm-2} (see Theorem 4.2 in
\cite{l:mrv18}) and the Lipschitz propagation of smallness stated in Proposition \ref{theo:LPS}. Let us notice that in the statement of Theorem
4.3 in \cite{l:mrv17} it was made the explicit assumption that the
domain contains a disc of radius $s_0\rho$, since this condition
plays a fundamental role in the proof and, consequently, the
constant $C$ appearing in the inequality \eqref{eq:korn_gener}
depends on $s_0$. This hypothesis, which was emphasized in
\cite{l:mrv17} because of its relevance for the
derivation of the estimate, can be deduced {}from the boundary
regularity, with $s_0=\frac{L}{L^2+1+\sqrt{L^2+1}}$ and
therefore it is omitted here.

\begin{proof}[Proof of Theorem \ref{theo:size_below_cavity}]
Let us consider $D^{r_D}\subset \Omega$ and its boundary $\partial D^{r_D}=\partial D \cup \Gamma^{r_D}$, where $\Gamma^{r_D} =\{x\in \Omega\setminus D \ |\ dist(x,\partial D)={r_D}\}$. Let us tessellate $\R^2$ with internally nonoverlapping closed squares having side
$l=\frac{{r_D}}{2\sqrt 2}$ and let $Q_1$, ..., $Q_N$ be those squares having nonempty intersection with $D^{r_D}$. Let us define
$\widetilde{D}^{r_D}$ the interior of $\cup_{i=1}^N Q_i \setminus D$. We have that $\partial \widetilde{D}^{r_D}=\partial D \cup \Sigma^{r_D}$, where $\Sigma^{r_D}\subset \cup_{j\in J}\partial Q_j$, with $J=\{j\ | \ Q_j\cap \Gamma^{r_D}\neq \emptyset\}$. As a portion of the boundary of
$\widetilde{D}^{r_D}$, $\partial D$ is of Lipschitz class with constants $\frac{{r_D}}{\sqrt{L_D^2+1}}$ and $L_D$. By construction, $\Sigma^{r_D}$ is of Lipschitz class with constants $\frac{{r_D}}{8}$ and $1$. Therefore $\partial \widetilde{D}^{r_D}$ is of Lipschitz class with constants $\gamma {r_D}$, $L'$, where $\gamma = \left(\max\left\{8,\sqrt{L_D^2+1}\right\}\right)^{-1}$ and $L'=\max\{1,L_D\}$.
Moreover, $\hbox{diam}(\widetilde{D}^{r_D})\leq (Q_D+3){r_D}$.
Let
\begin{equation*}
   x_{\partial D} =\frac{1}{|\partial D|}
    \int_{\partial D} x
\end{equation*}
be the center of mass of $\partial D$. Let
\begin{equation*}
   a=\frac{1}{|\partial D|}
    \int_{\partial D} w_c,
\end{equation*}
\begin{equation*}
   b=\frac{1}{|\partial D|}
    \int_{\partial D} \varphi_c,
\end{equation*}
\begin{equation*}
   W=\frac{1}{2|\widetilde{D}^{r_D}|}
    \int_{\widetilde{D}^{r_D}} \nabla \varphi_c - \nabla^T \varphi_c,
\end{equation*}
\begin{equation*}
   r=b+W(x-x_{\partial D}),
\end{equation*}
\begin{equation*}
   \varphi_c^*=\varphi_c - r,
\end{equation*}
\begin{equation*}
   w_c^*=w_c +b\cdot(x-x_{\partial D}) +a.
\end{equation*}
By these definitions, $\varphi_c^*$ and $w_c^*$ have zero mean on $\partial D$, and
\begin{equation*}
   \varphi_c^* + \nabla w_c^*=\varphi_c + \nabla w_c -W(x-x_{\partial D}).
\end{equation*}
By the weak formulation of the Reissner-Mindlin system satisfied by $(\varphi_0,w_0)$ in $D$ choosing the test pair $(\varphi_c^*,w_c^*)$, and recalling the right hand side of \eqref{eq:double_inequality}, we have
\begin{multline}
  \label{eq:proof_cavity_below1}
W_c-W_0=\int_{D} {\mathbb P}\nabla \varphi_0\cdot \nabla \varphi_c^* + \int_{D}
S(\varphi_0+\nabla w_0)\cdot(\varphi_c^*+\nabla w_c^*+W(x-x_{\partial D}))=\\
=\int_{\partial D}({\mathbb P}\nabla \varphi_0 n)\cdot \varphi_c^* +
\int_{\partial D}(S(\varphi_0+\nabla w_0)\cdot n)w_c^*
+\int_D S(\varphi_0+\nabla w_0)\cdot W(x-x_{\partial D})=
\\
=I_1+I_2+I_3.
\end{multline}
By applying H\"older inequality and by \eqref{eq:lip_operatori},
\begin{equation}
  \label{eq:proof_cavity_below2}
|I_1|\leq C\left(h^3\int_{\partial  D}|\nabla \varphi_0|^2\right)^\frac{1}{2}
\left(h^3\int_{\partial  D}|\varphi_c^*|^2\right)^\frac{1}{2},
\end{equation}
with $C$ only depending on $\alpha_0$, $\gamma_0$, $\alpha_1$.

Recalling \eqref{eq:D-away_boundary}, we can apply interior
regularity estimates (see \cite[Theorem 6.1]{l:c80}) and then, by taking into account the normalization condition \eqref{eq:normalization-0}, by applying Proposition
\ref{prop:Korn_gener} to $(\varphi_0,w_0)$ in $\Omega$, and
recalling \eqref{eq:E_double} and \eqref{eq:W_0}, we have
\begin{multline}
  \label{eq:proof_cavity_below3}
h^3\int_{\partial  D}|\nabla \varphi_0|^2\leq h^3|\partial D|\|\nabla\varphi_0\|_{L^\infty(D)}^2
\leq Ch^3|\partial D|
\left(
\|\varphi_0\|_{H^1(\Omega)}^2+\frac{1}{\rho_0^2}\|w_0\|_{H^1(\Omega)}^2
\right)
\leq\\
\leq
Ch^3|\partial D|\left(\|\widehat{\nabla} \varphi_0\|_{L^2(\Omega)}^2+\frac{1}{\rho_0^2}
\|\varphi_0+\nabla w_0\|_{L^2(\Omega)}^2\right)\leq \frac{C}{\rho_o^2}|\partial D| W_0,
\end{multline}
with $C$ only depending on $\alpha_0$, $\gamma_0$, $\alpha_1$, $\frac{\rho_0}{h}$, $L_0$, $Q_0$, $d_0$.

By \eqref{eq:PoincGr}, \eqref{eq:compkorn} and \eqref{eq:E_double} we have
\begin{multline}
  \label{eq:proof_cavity_below4}
h^3\int_{\partial  D}|\varphi_c^*|^2\leq Ch^3{r_D} \int_{D^{r_D}}|\nabla \varphi_c^*|^2
\leq Ch^3{r_D} \int_{\widetilde{D}^{r_D}}|\nabla \varphi_c^*|^2 \leq
\\
\leq
Ch^3{r_D} \int_{\widetilde{D}^{r_D}}|\widehat{\nabla} \varphi_c|^2
\leq C {r_D}W_c,
\end{multline}
with $C$ only depending on $\alpha_0$, $\gamma_0$, $\alpha_1$, $\frac{\rho_0}{h}$, $L_D$, $Q_D$.

By using arguments similar to those in \cite[Lemma 2.8]{l:ar98}, we have that
\begin{equation}
  \label{eq:proof_cavity_below5}
|\partial D|\leq C \frac{|D|}{{r_D}},
\end{equation}
with $C$ only depending on $L_D$.

By \eqref{eq:proof_cavity_below2}--\eqref{eq:proof_cavity_below5},
\begin{equation}
  \label{eq:proof_cavity_below6}
|I_1|\leq \frac{C}{\rho_0}|D|^{\frac{1}{2}}W_0^{\frac{1}{2}}W_c^{\frac{1}{2}},
\end{equation}
with $C$ only depending on $\alpha_0$, $\gamma_0$, $\alpha_1$, $\frac{\rho_0}{h}$, $L_0$, $Q_0$, $L_D$, $Q_D$, $d_0$.

Let us estimate $I_3$. By H\"older inequality and by \eqref{eq:lip_operatori},
\begin{multline}
  \label{eq:proof_cavity_below7}
|I_3|\leq Ch\left(\int_D|\varphi_0+\nabla w_0|^2\right)^{\frac{1}{2}}
\left(\int_D|W(x-x_{\partial D})|^2\right)^{\frac{1}{2}}\leq\\
\leq
Ch\|\varphi_0+\nabla w_0\|_{L^\infty(D)}|D|^{\frac{1}{2}}|W|\left(\int_D|x-x_{\partial D}|^2\right)^{\frac{1}{2}},
\end{multline}
with $C$ only depending on $\alpha_1$.

By interior regularity estimates, by using the normalization
conditions \eqref{eq:normalization-0}, the Poincar\'{e} inequality
\eqref{eq:PoincG-a}, the Generalized Korn inequality
\eqref{eq:korn_gener}, and recalling \eqref{eq:E_double} and
\eqref{eq:W_0}, we have
\begin{multline}
  \label{eq:proof_cavity_below8}
h\|\varphi_0+\nabla w_0\|_{L^\infty(D)}\leq
 Ch\
\left(
\|\varphi_0\|_{H^1(\Omega)}+ \frac{1}{\rho_0}  \| w_0\|_{H^1(\Omega)}\right)\leq\\
\leq Ch\left(\int_\Omega
|\widehat{\nabla} \varphi_0|^2+\frac{1}{\rho_0^2}|\varphi_0+\nabla w_0|^2\right)^{\frac{1}{2}}\leq \frac{C}{\sqrt{\rho_0}}W_0^{\frac{1}{2}},
\end{multline}
with with $C$ only depending on $\alpha_0$, $\gamma_0$, $\alpha_1$, $\frac{\rho_0}{h}$, $d_0$, $L_0$, $Q_0$.

By H\"older inequality, by the Generalized Korn inequality
\eqref{eq:korn_gener}, noticing that $|\widetilde{D}^{r_D}|\geq
C{r_D}^2$, with $C$ only depending on $L_D$, using
$r_D<\frac{d_0}{2}\rho_0$, and recalling \eqref{eq:E_double} and
\eqref{eq:W_c}, we have
\begin{multline}
  \label{eq:proof_cavity_below9}
|W|\leq \frac{C}{|\widetilde{D}^{r_D}|^{\frac{1}{2}}}\left(\int_{\widetilde{D}^{r_D}}|\nabla \varphi_c|^2\right)^{\frac{1}{2}}
\leq \frac{C}{{r_D}} \left(\int_{\widetilde{D}^{r_D}}|\widehat{\nabla} \varphi_c|^2
+\frac{1}{{r_D}^2}|\varphi_c+\nabla w_c|^2\right)^{\frac{1}{2}}\leq \\
\leq \frac{C}{{r_D}^2} \left(\int_{\widetilde{D}^{r_D}}\rho_0^2|\widehat{\nabla} \varphi_c|^2
+|\varphi_c+\nabla w_c|^2\right)^{\frac{1}{2}}\leq \frac{C}{{r_D}^2\sqrt{\rho_0}}W_c^{\frac{1}{2}},
\end{multline}
with $C$ only depending on $\alpha_0$, $\gamma_0$, $\alpha_1$, $\frac{\rho_0}{h}$, $L_D$, $Q_D$, $d_0$.
Moreover,
\begin{equation}
  \label{eq:proof_cavity_below10}
\left(\int_D|x-x_{\partial D}|^2\right)^{\frac{1}{2}}\leq |D|^{\frac{1}{2}}
\hbox{diam}(D)\leq C{r_D}^2
\end{equation}
with $C$ only depending on $Q_D$.

By \eqref{eq:proof_cavity_below7}--\eqref{eq:proof_cavity_below10}, it follows that
\begin{equation}
  \label{eq:proof_cavity_below11}
|I_3|\leq \frac{C}{\rho_0}W_0^{\frac{1}{2}}W_c^{\frac{1}{2}}|D|^{\frac{1}{2}},
\end{equation}
with with $C$ only depending on $\alpha_0$, $\gamma_0$, $\alpha_1$, $\frac{\rho_0}{h}$, $d_0$, $L_0$, $Q_0$, $L_D$, $Q_D$.

By applying H\"older inequality, by \eqref{eq:lip_operatori}, \eqref{eq:PoincGr}, \eqref{eq:proof_cavity_below5} and \eqref{eq:proof_cavity_below8}, we get
\begin{equation}
  \label{eq:proof_cavity_below12}
|I_2|\leq \frac{C}{\sqrt{\rho_0}}W_0^{\frac{1}{2}}|D|^{\frac{1}{2}}\left(\int_{\widetilde{D}^{r_D}}|\nabla w_c^*|^2\right)^{\frac{1}{2}},
\end{equation}
with with $C$ only depending on $\alpha_0$, $\gamma_0$, $\alpha_1$, $\frac{\rho_0}{h}$, $d_0$, $L_0$, $Q_0$, $L_D$, $Q_D$.

On the other hand,
\begin{multline}
  \label{eq:proof_cavity_below13}
\int_{\widetilde{D}^{r_D}}|\nabla w_c^*|^2= \int_{\widetilde{D}^{r_D}}|\nabla w_c+b|^2\leq
2\int_{\widetilde{D}^{r_D}}|\nabla w_c+\varphi_c|^2+2\int_{\widetilde{D}^{r_D}}|\varphi_c-b|^2\leq
\\
\leq
2\int_{\widetilde{D}^{r_D}}|\nabla w_c+\varphi_c|^2+4\int_{\widetilde{D}^{r_D}}|\varphi_c-b-W(x-x_{\partial D})|^2+4\int_{\widetilde{D}^{r_D}}|W(x-x_{\partial D})|^2.
\end{multline}
By applying the Poincar\'{e} inequality \eqref{eq:PoincG-a-bis} and the Korn inequality
\eqref{eq:korn2}, using \eqref{eq:E_double} and \eqref{eq:proof_cavity_below9} and recalling that
$\hbox{diam}(\widetilde{D}^{r_D})\leq (Q_D+3){r_D}$ and $|\widetilde{D}^{r_D}|\leq C{r_D}^2$, with $C$ only depending on $Q_D$, we have
\begin{multline}
  \label{eq:proof_cavity_below14}
\int_{\widetilde{D}^{r_D}}|\nabla w_c^*|^2\leq
C\int_{\widetilde{D}^{r_D}}|\nabla w_c+\varphi_c|^2+C{r_D}^2\int_{\widetilde{D}^{r_D}}|\widehat{\nabla} \varphi_c|^2+C|W|^2\int_{\widetilde{D}^{r_D}}|x-x_{\partial D}|^2\leq
\\
\leq C\int_{\widetilde{D}^{r_D}}(|\varphi_c+\nabla w_c|^2+\rho_0^2|\widehat{\nabla} \varphi_c|^2)+\frac{C}{\rho_0}W_c\leq \frac{C}{\rho_0}W_c,
\end{multline}
where $C$ only depends on $\alpha_0$, $\gamma_0$, $\alpha_1$, $\frac{\rho_0}{h}$, $d_0$, $L_D$, $Q_D$.

By \eqref{eq:proof_cavity_below12} and \eqref{eq:proof_cavity_below14}, it follows that
\begin{equation}
  \label{eq:proof_cavity_below15}
    |I_2|\leq \frac{C}{\rho_0}W_0^{\frac{1}{2}}W_c^{\frac{1}{2}}|D|^{\frac{1}{2}},
\end{equation}
where $C$ only depends on $\alpha_0$, $\gamma_0$, $\alpha_1$, $\frac{\rho_0}{h}$, $d_0$, $L_0$, $Q_0$, $L_D$, $Q_D$.

Finally, by \eqref {eq:proof_cavity_below1}, \eqref{eq:proof_cavity_below6}, \eqref{eq:proof_cavity_below12} and \eqref{eq:proof_cavity_below15},
\begin{equation}
  \label{eq:proof_cavity_below16}
    W_c-W_0\leq \frac{C}{\rho_0}W_0^{\frac{1}{2}}W_c^{\frac{1}{2}}|D|^{\frac{1}{2}},
\end{equation}
where $C$ only depends on $\alpha_0$, $\gamma_0$, $\alpha_1$, $\frac{\rho_0}{h}$, $d_0$, $L_0$, $Q_0$, $L_D$, $Q_D$, and the thesis follows by straightforward calculations.
\end{proof}

\section{Proof of the size estimates for rigid inclusions}
\label{SE-proof_rigid}

The comparison between the works $W_0$ and
$W_r$ is stated in the following lemma.
\begin{lem}
   \label{lem:double-inequality-rigid}
Let $\Omega$ be a bounded domain in $\R^2$ and $D \subset \subset \Omega$ a measurable set. Let $S$, $\mathbb P$ given in
\eqref{eq:shearing-tensor}, \eqref{eq:bending-tensor} satisfy the
strong convexity conditions \eqref{eq:Lame-ell}. Let $(\varphi_r, w_r) \in H_D^1(\Omega,
\R^2) \times H_D^1(\Omega)$, $(\varphi_0, w_0) \in H^1(\Omega,
\R^2) \times H^1(\Omega)$ be the weak solutions to problems \eqref{eq:1.dir-pbm-rig-1}--\eqref{eq:1.dir-pbm-rig-8} and \eqref{eq:1.dir-pbm-1}--\eqref{eq:1.dir-pbm-4}, respectively. We
have
\begin{multline}
  \label{eq:double-inequality-rigid}
    \int_{D}  {\mathbb P}\nabla \varphi_0\cdot \nabla \varphi_0 +
    S(\varphi_0+\nabla w_0)\cdot
        (\varphi_0+\nabla w_0)
    \leq W_0-W_r=\\
    =\int_{\partial D}
    (\mathbb P \nabla \varphi_r^+) n \cdot \varphi_0 +
    (S(\varphi_r^+ + \nabla w_r^+)\cdot n) w_0.
\end{multline}
\end{lem}
\begin{proof}
The proof of this \textit{energy lemma} can be obtained by
adapting the proof of the corresponding result in linear
elasticity, see \cite{l:mr03}. Therefore, we skip the details and
we report the main steps of the proof.

Let us multiply equations \eqref{eq:1.dir-pbm-rig-1},
\eqref{eq:1.dir-pbm-rig-2} by $w_0$, $\varphi_0$, respectively.
Integrating by parts and summing up, we obtain
\begin{multline}
  \label{eq:rigid-7-2}
    \int_{\Omega \setminus \overline{D}}  {\mathbb P}\nabla \varphi_r^+ \cdot \nabla \varphi_0 +
    S(\varphi_r^+ +\nabla w_r^+)\cdot
        (\varphi_0+\nabla w_0)  = \\
        = W_0 - \int_{\partial D}
    (\mathbb P \nabla \varphi_r^+) n \cdot \varphi_0 +
    (S(\varphi_r^+ + \nabla w_r^+)\cdot n) w_0.
\end{multline}
Next, we multiply equations \eqref{eq:1.dir-pbm-1},
\eqref{eq:1.dir-pbm-2} by $w_r^+$, $\varphi_r^+$, respectively,
and we integrate by parts in $\Omega \setminus \overline{D}$. Summing up, we
obtain
\begin{multline}
  \label{eq:rigid-8-1}
    W_r = \int_{\Omega \setminus \overline{D}}  {\mathbb P}\nabla \varphi_r^+ \cdot \nabla \varphi_0 +
    S(\varphi_r^+ +\nabla w_r^+)\cdot
        (\varphi_0+\nabla w_0)  - \\
        - \int_{\partial D}
    (\mathbb P \nabla \varphi_0) n \cdot \varphi_r^+ +
    (S(\varphi_0 + \nabla w_0)\cdot n) w_r^+=
    \\
    =  \int_{\Omega \setminus \overline{D}}  {\mathbb P}\nabla \varphi_r^+ \cdot \nabla \varphi_0 +
    S(\varphi_r^+ +\nabla w_r^+)\cdot
        (\varphi_0+\nabla w_0) ,
\end{multline}
where, in the last step, we have used the fact that the second
integral on the right hand side vanishes because of the definition
of $(\varphi_r, w_r)$ in $\overline{D}$. By \eqref{eq:rigid-7-2}
and \eqref{eq:rigid-8-1}, the equality on the right hand side of
\eqref{eq:double-inequality-rigid} follows.

To obtain the inequality in \eqref{eq:double-inequality-rigid}, we
consider the quadratic form of the strain energy associated to the
pair $(\varphi_0 -\varphi_r, w_0 - w_r)$ in $\Omega$. Recalling
the definition of $(\varphi_r^-, w_r^-)$ in $\overline{D}$, by
\eqref{eq:work-Wr}, \eqref{eq:W_0} and by \eqref{eq:rigid-8-1}, we
have
\begin{multline}
  \label{eq:rigid-10-1}
    \int_\Omega  \mathbb P \nabla (\varphi_0 - \varphi_r) \cdot \nabla(\varphi_0 -
    \varphi_r)+ S((\varphi_0-\varphi_r) + \nabla (w_0-w_r)) \cdot ((\varphi_0-\varphi_r) + \nabla
    (w_0-w_r)) =
    \\
    = \int_\Omega  \mathbb P \nabla \varphi_0  \cdot \nabla \varphi_0 +
     S(\varphi_0 + \nabla w_0) \cdot (\varphi_0 + \nabla w_0)
    +
    \\
    +
    \int_{\Omega \setminus \overline{D}}  \mathbb P \nabla \varphi_r^+  \cdot \nabla \varphi_r^+ +
     S(\varphi_r^+ + \nabla w_r^+) \cdot (\varphi_r^+ + \nabla w_r^+)
     -
     \\
     -
    2 \int_{\Omega \setminus \overline{D}}  \mathbb P \nabla \varphi_0  \cdot \nabla \varphi_r^+ +
     S(\varphi_0 + \nabla w_0) \cdot (\varphi_r^+ + \nabla w_r^+)
     = W_0 - W_r.
\end{multline}
Noticing that
\begin{multline}
  \label{eq:rigid-10-3}
    \int_D  \mathbb P \nabla \varphi_0 \cdot \nabla \varphi_0 +
    S(\varphi_0 + \nabla w_0) \cdot (\varphi_0 + \nabla w_0)
    =
    \\
    =\int_D  \mathbb P \nabla (\varphi_0 - \varphi_r^-) \cdot \nabla(\varphi_0 -
    \varphi_r^-)+ S((\varphi_0-\varphi_r^-) + \nabla (w_0-w_r^-)) \cdot ((\varphi_0-\varphi_r^-) + \nabla
    (w_0-w_r^-))
    \leq
    \\
    \leq
    \int_\Omega  \mathbb P \nabla (\varphi_0 - \varphi_r) \cdot \nabla (\varphi_0 -
    \varphi_r)+ S((\varphi_0-\varphi_r) + \nabla (w_0-w_r)) \cdot ((\varphi_0-\varphi_r) + \nabla
    (w_0-w_r)),
\end{multline}
by \eqref{eq:rigid-10-1} the thesis follows.
\end{proof}
Let us notice that the estimate {}from above stated in Theorem
\ref{theo:size-above-rigid} can be derived as in the proof of Theorem
\ref{theo:size_above_cavity}.

In order to prove Theorem \ref{theo:size-below-rigid} we shall use
the following proposition.
\begin{prop}
  \label{prop:size-below-rigid-L2-contact-actions}
Let the hypotheses of Theorem \ref{theo:size-below-rigid} be
satisfied. The contact actions exerted by the material in $\Omega
\setminus \overline{D}$ on $D$ throughout the boundary $\partial
D$ are square summable on $\partial D$, e.g., $(\mathbb P \nabla
\varphi_r^+) n \in L^2(\partial D, \R^2)$ and $S(\varphi_r^+ +
\nabla w_r^+) \cdot n \in L^2(\partial D)$, and the following
estimate holds
\begin{equation}
  \label{eq:rigid-L2-estimate-partialD}
    \int_{\partial D} |(\mathbb P \nabla
    \varphi_r^+) n|^2 + \rho_0^2 | S(\varphi_r^+ + \nabla w_r^+) \cdot n|^2 \leq
    C  \frac{\rho_0}{r_D}  \int_{\Omega \setminus \overline{D}} \rho_0^5 |
    \widehat{\nabla} \varphi_r^+ |^2 + \rho_0^3 | \varphi_r^+ + \nabla w_r^+|^2,
\end{equation}
where $n$ denotes the outer unit normal to $D$ and the constant $C>0$ only depends on $Q_0$, $d_0$, $L_D$,
$Q_D$, $\alpha_0$, $\alpha_1$, $\gamma_0$.
\end{prop}
\begin{proof}[Proof of Theorem \ref{theo:size-below-rigid}]

By using \eqref{eq:rigid-1-10} and \eqref{eq:rigid-1-11}, the
right hand side of \eqref{eq:double-inequality-rigid} can be
written as
\begin{multline}
  \label{eq:rigid-11-1}
    W_0-W_r
     =
    \int_{\partial D}
    ((\mathbb P \nabla \varphi_r^+)n  -  (S(\varphi_r^+ + \nabla w_r^+)\cdot n)x) \cdot
    (\varphi_0 - \varphi_{0, \partial D})
    +
    \\
    +
     \int_{\partial D}
     (S(\varphi_r^+ + \nabla w_r^+)\cdot n)x \cdot \varphi_0 +
     \int_{\partial D}
     (S(\varphi_r^+ + \nabla w_r^+)\cdot n) (w_0 - w_{0,\partial D})
     \equiv I_1 + I_2 + I_3.
\end{multline}
By applying H\"{o}lder's inequality and Poincar\'{e}'s inequality
\eqref{eq:PoincGr} we have
\begin{equation}
  \label{eq:rigid-12-1}
    |I_1|
    \leq
    C r_D^{1/2}
    \left (
    \int_D | \nabla \varphi_0 |^2
    \right )^{1/2}
    \left (
    \int_{\partial D}
    |(\mathbb P \nabla \varphi_r^+)n|^2 + \rho_0^2 |S(\varphi_r^+ + \nabla w_r^+) \cdot n|^2
    \right )^{1/2},
\end{equation}
where $C>0$ only depends on $Q_0$, $L_D$, $Q_D$.

By interior regularity estimates, by the generalized Korn inequality
\eqref{eq:korn_gener} (applied to $(\varphi_0,w_0)$ in
$\Omega$), and by recalling \eqref{eq:E_double} and \eqref{eq:W_0}, we
have
\begin{equation}
  \label{eq:rigid-13-1}
    \left (
    \int_D | \nabla \varphi_0 |^2
    \right )^{1/2}
    \leq
    \frac{C}{\rho_0^{5/2} } |D|^{1/2} W_0^{1/2},
\end{equation}
where $C>0$ only depends on $\alpha_0$, $\gamma_0$, $\alpha_1$,
$\frac{\rho_0}{h}$, $L_0$, $Q_0$, $d_0$. Therefore, by
\eqref{eq:rigid-12-1} and \eqref{eq:rigid-13-1}, we have
\begin{equation}
  \label{eq:rigid-13-2}
    |I_1|
    \leq
    \frac{C}{\rho_0^{5/2} } r_D^{1/2} |D|^{1/2} W_0^{1/2}
    \left (
    \int_{\partial D}
    |(\mathbb P \nabla \varphi_r^+)n|^2 + \rho_0^2|S(\varphi_r^+ + \nabla w_r^+) \cdot n|^2
    \right )^{1/2},
\end{equation}
where the constant $C>0$ only depends on $\alpha_0$, $\gamma_0$,
$\alpha_1$, $\frac{\rho_0}{h}$, $L_0$, $Q_0$, $d_0$, $L_D$, $Q_D$.

By using similar estimates, we get
\begin{equation}
  \label{eq:rigid-13-3}
    |I_3|
    \leq
    \frac{C}{\rho_0^{5/2} } r_D^{1/2} |D|^{1/2} W_0^{1/2}
    \left (
    \int_{\partial D}
    \rho_0^2  |S(\varphi_r^+ + \nabla w_r^+) \cdot n|^2
    \right )^{1/2},
\end{equation}
where the constant $C>0$ only depends on $\alpha_0$, $\gamma_0$,
$\alpha_1$, $\frac{\rho_0}{h}$, $L_0$, $Q_0$, $d_0$, $L_D$, $Q_D$.

By \eqref{eq:rigid-1-10} and by using H\"{o}lder's inequality, the
integral $I_2$ can be dominated as follows
\begin{multline}
  \label{eq:rigid-13-4}
    I_2
    =
    \int_{\partial D}
     (S(\varphi_r^+ + \nabla w_r^+)\cdot n) (x \cdot \varphi_0 -
     ( x \cdot \varphi_0  )_{\partial D})
     \leq
     \\
     \leq
     \left (
     \int_{\partial D}
     | x \cdot \varphi_0 - ( x \cdot \varphi_0  )_{\partial D} |^2
     \right )^{1/2}
     \left (
    \int_{\partial D}
    |S(\varphi_r^+ + \nabla w_r^+) \cdot n|^2
    \right )^{1/2}.
\end{multline}
Noticing that $\nabla( x \cdot \varphi_0) = \varphi_0 + (\nabla
\varphi_0)^T x$, the first integral on the right hand side of
\eqref{eq:rigid-13-4} can be estimated by using Proposition
\ref{prop:Poinc}, interior regularity estimates for $\nabla
\varphi_0$, the generalized Korn's inequality
\eqref{eq:korn_gener} (applied to $(\varphi_0,w_0)$ in $\Omega$),
inequality \eqref{eq:E_double} and the definition of $W_0$ in
\eqref{eq:W_0}, obtaining
\begin{multline}
  \label{eq:rigid-14-1}
     \int_{\partial D}
     | x \cdot \varphi_0 - ( x \cdot \varphi_0  )_{\partial D} |^2
     \leq
     C r_D \int_{D} |\nabla(x \cdot \varphi_0)|^2
     \leq
     C r_D \int_D | \varphi_0|^2 + |x|^2 |\nabla \varphi_0|^2
     \leq
     \\
     \leq
     C r_D |D| ( \|\varphi_0\|_{L^\infty(D)}^2 + \rho_0^2 \|\nabla
     \varphi_0\|_{L^\infty(D)}^2)
     \leq
     \frac{C}{\rho_0^{3} } r_D |D| W_0,
\end{multline}
where $C>0$ only depends on $\alpha_0$, $\gamma_0$, $\alpha_1$,
$\frac{\rho_0}{h}$, $L_0$, $Q_0$, $d_0$, $L_D$, $Q_D$. Inserting
the above estimate in \eqref{eq:rigid-13-4} we have
\begin{equation}
  \label{eq:rigid-14-2}
    I_2 \leq
    \frac{C}{\rho_0^{3/2} } r_D^{1/2} |D|^{1/2} W_0^{1/2}
     \left (
    \int_{\partial D}
    |S(\varphi_r^+ + \nabla w_r^+) \cdot n|^2
    \right )^{1/2},
\end{equation}
where the constant $C>0$ only depends on $\alpha_0$, $\gamma_0$,
$\alpha_1$, $\frac{\rho_0}{h}$, $L_0$, $Q_0$, $d_0$, $L_D$, $Q_D$.

By \eqref{eq:rigid-11-1}, \eqref{eq:rigid-13-2},
\eqref{eq:rigid-13-3}, \eqref{eq:rigid-14-2} and by Proposition
\ref{prop:size-below-rigid-L2-contact-actions}, we have
\begin{equation}
  \label{eq:rigid-14-3}
    W_0-W_r
    \leq
    \frac{C}{\rho_0^{3/2} } |D|^{1/2} W_0^{1/2}
    \left (
    \int_{\Omega \setminus \overline{D}} \rho_0^4 |
    \widehat{\nabla} \varphi_r^+ |^2 + \rho_0^2 | \varphi_r^+ + \nabla w_r^+|^2
    \right )^{1/2},
\end{equation}
with $C>0$ only depending on $\alpha_0$, $\gamma_0$, $\alpha_1$,
$\frac{\rho_0}{h}$, $L_0$, $Q_0$, $d_0$, $L_D$, $Q_D$.

To conclude, by the strong convexity of $\mathbb P$ and
$S$, recalling \eqref{eq:E_double} and \eqref{eq:work-Wr}, we have
\begin{multline}
  \label{eq:rigid-31-2}
   \int_{\Omega \setminus \overline{D}} \rho_0^4 |
    \widehat{\nabla} \varphi_r^+ |^2 + \rho_0^2 | \varphi_r^+ + \nabla w_r^+|^2
   \leq
   \\
    \leq
    C \rho_0 \int_{\Omega \setminus \overline{D}}
    \mathbb P \nabla \varphi_r^+ \cdot \nabla \varphi_r^+ + S
    (\varphi_r^+ + \nabla w_r^+) \cdot (\varphi_r^+ + \nabla w_r^+)
    =
    C \rho_0 W_r,
\end{multline}
with $C>0$ only depending on $\alpha_0$, $\gamma_0$, $\alpha_1$,
$\frac{\rho_0}{h}$, $L_0$, $Q_0$, $d_0$, $L_D$, $Q_D$. Therefore,
by \eqref{eq:rigid-14-3} and \eqref{eq:rigid-31-2}, we have
\begin{equation}
  \label{eq:rigid-final}
    W_0-W_r
    \leq
    \frac{C}{\rho_0 } |D|^{1/2} W_0^{1/2} W_r^{1/2},
\end{equation}
with $C>0$ only depending on $\alpha_0$, $\gamma_0$, $\alpha_1$,
$\frac{\rho_0}{h}$, $L_0$, $Q_0$, $d_0$, $L_D$, $Q_D$. By some
algebra, estimate \eqref{eq:size-below-rigid} follows.
\end{proof}
The remaining part of the section is devoted to the proof of
Proposition \ref{prop:size-below-rigid-L2-contact-actions}. The
main idea consists in estimating the $L^2(\partial D)$-norm of the
conormal derivatives $(\mathbb P \nabla \varphi_r^+) n$, $S\nabla
w_r^+ \cdot n$ in terms of the strain energy stored in $\Omega
\setminus \overline{D}$ and the $L^2(\partial D)$-norm of the
tangential component of the gradient of $\varphi_r^+$ and $w_r^+$.

We start by introducing some notation.

Given $ \rho > 0$, $L > 0$ and a Lipschitz continuous function
$\psi : (-2\rho, 2\rho) \rightarrow \R$ satisfying
$\psi (0)=0$, $\| \psi \|_{C^{0,1}((-2\rho, 2\rho))} \leq 2 \rho
L$, let us define for every $t$, $0<t\leq2 \rho$,
\begin{equation}
  \label{eq:set-cyl+}
  C_{t}^{+}
  =
  \{(x_1,x_2) \in \R^2 \ | \ |x_1|<t, \
\psi(x_1) < x_2 < L t \},
\end{equation}
\begin{equation}
  \label{eq:set-delta}
  \Delta_{t}
  =
  \{(x_1,x_2) \in \R^2 \ | \ |x_1|<t, \ x_2=\psi(x_1)  \}.
\end{equation}
We shall use the following two-dimensional version of the
constructive Korn-type inequality on cylindrical domains due to
Kondrat'ev and Oleinik \cite{l:ko}.
\begin{prop}
  \label{prop:Korn-Kondratev-Oleinik} {\rm(\cite{l:ko}, Theorem 2)}
  Let
\begin{equation}
  \label{eq:6.0c}
  C_{l', l}
  =
  \{(x_1,x_2) \in \R^2 \ | \ |x_1|<l', \
-l < x_2 < l \},
\end{equation}
where $l>l'$. For every $u \in H^1(C_{l', l},\R^2)$ such that
$u=0$ on $\{ x_2=-l \}$, we have
\begin{equation}
  \label{eq:Korn-Kondratev-Oleinik}
  \int_{C_{l', l}}
  | \nabla u |^2
  \leq
  C
  \left ( 1+ \frac {4l^2} {l'^{2}}
  \right )
  \int_{C_{l', l}}
  | \widehat {\nabla} u |^2,
\end{equation}
where $C>0$ is an absolute constant.
\end{prop}
The next proposition states local boundary estimates in $L^2$ of
the conormal derivatives of solutions to the Mindlin-Reissner
plate problem. A proof shall be presented at the end of this
section.
\begin{prop}
  \label{prop:Local-Boundary-Estimates-L2}
Let $S$, $\mathbb P$ given in \eqref{eq:shearing-tensor},
\eqref{eq:bending-tensor} satisfy the strong convexity conditions
\eqref{eq:Lame-ell}.

Let $w \in H^{1}(C_{2 \rho}^{+})$ be a solution to
\begin{equation}
  \label{eq:rigid-17-1}
  {\rm div} (S \nabla w) = - {\rm div} (S\varphi) \quad \hbox{ in }\ C_{2
  \rho}^{+},
\end{equation}
with $\varphi \in H^1 (C_{2
  \rho}^{+}, \R^2)$.

If $w|_{\Delta_{2 \rho}} \in H^{1}(\Delta_{2 \rho})$, then $S
\nabla w \cdot n  \in L^{2}(\Delta_{\rho})$ and we have
\begin{equation}
  \label{eq:rigid-17-2}
  \int_{\Delta_{\rho}}
  |S \nabla w \cdot n|^{2}
  \leq
  C
  \left (
  \int_{\Delta_{2 \rho}} \rho_0^2 |\nabla_{T}w |^{2}
  +
  \left ( 1+\frac{\rho_0}{\rho} \right ) \int_{C_{2 \rho}^{+}} \rho_0 |\nabla w|^{2}
    + \int_{C_{2 \rho}^{+}} \rho_0 |\varphi|^{2} + \rho_0^3 |\nabla \varphi|^2
  \right )
\end{equation}
where $\nabla_T w$ is the tangential component of $\nabla w$, and
the constant $C>0$ only depends on $L$, $\alpha_0$, $\gamma_0$,
$\alpha_1$.

Let $\varphi \in H^{1}(C_{2 \rho}^{+}, \R^2)$ be a solution to
\begin{equation}
  \label{eq:rigid-18-1}
  {\rm div} (\mathbb P \nabla \varphi)  =  S(\varphi + \nabla w) \quad \hbox{ in }\ C_{2
  \rho}^{+},
\end{equation}
with $w \in H^1 (C_{2
  \rho}^{+})$.

If $\varphi|_{\Delta_{2 \rho}} \in H^{1}(\Delta_{2 \rho}, \R^2)$,
then $(\mathbb P \nabla \varphi ) n  \in L^{2}(\Delta_{\rho},
\R^2)$ and we have
\begin{equation}
  \label{eq:rigid-18-2}
  \int_{\Delta_{\rho}}
  |(\mathbb P \nabla \varphi ) n|^{2}
  \leq
  C
  \left (
  \int_{\Delta_{2 \rho}} \rho_0^6 |\nabla_{T} \varphi |^{2}
  +
  \left ( 1+\frac{\rho_0}{\rho} \right ) \int_{C_{2 \rho}^{+}} \rho_0^3 |\varphi |^2 + \rho_0^5 |\nabla
  \varphi|^{2}  + \int_{C_{2 \rho}^{+}}  \rho_0^3 |\nabla w|^2
  \right )
\end{equation}
where $\nabla_T \varphi$ is the tangential component of $\nabla
\varphi$, and the constant $C>0$ only depends on $L$,
$\alpha_0$, $\gamma_0$, $\alpha_1$.
\end{prop}
\begin{proof}[Proof of Proposition
\ref{prop:size-below-rigid-L2-contact-actions}]
We follow the lines of the proof derived in \cite {l:amr02}
(Proposition $3.4$) for the analogous estimate in an electric
conductor, see also \cite{l:mr03} (Proposition $3.3$).

We cover $\partial D$ with internally non-overlapping closed cubes
$Q_{j}$, $j=1,...,J$, having side $\widetilde{r}_D=\gamma(L_D)r_D$,
where
$\gamma(L_D)=\frac{\min\{1,L_D\}}{2\sqrt{2}\sqrt{1+L_D^{2}}}$. The
number of these cubes can be evaluated by a slight modification of
the arguments in Lemma $2.8$ of \cite{l:ar98}, that is
\begin{equation}
  \label{eq:rigid-25-2}
  J
  \leq
  C\frac{|D|}{r_D^2}
  \leq
  CQ_D^2,
\end{equation}
where $C>0$ only depends on $L_D$.

For every $j=1,...,J$ there exists $x_{0} \in \partial D \cap
Q_{j}$ such that $Q_{j} \cap (\Omega \setminus \overline{D})
\subset C_{\overline{r}}^{+}$, where $ \overline{r} =
\frac{r_D}{2\sqrt{1+L_D^{2}}}$ and $C_{t}^{+} = \{ y=(y_1, y_2)
\in \Omega \setminus \overline{D} \ | \ |y_1|<t,
\psi(y_1)<y_{2}<tL_D\}$ for every $t$, $0<t\leq 2\overline{r}$.
Here, $\psi$ is a Lipschitz function in $(-2\overline{r},
2\overline{r})$ satisfying $\psi(0)=0$ and
$\|\psi\|_{C^{0,1}((-2\overline{r}, 2\overline{r}))}\leq
2\overline{r}L_D$, representing locally the boundary of $D$ in a
suitable coordinate system $y=(y_{1},y_{2})$, $y=Rx$, where $R$ is
an orthogonal transformation and $x=(x_{1},x_{2})$ is the
referential cartesian coordinate system.

Recalling \eqref{eq:1.dir-pbm-rig-5}--\eqref{eq:1.dir-pbm-rig-8},
it is not restrictive to choose $(\varphi_r, w_r)$ such that
\begin{equation}
  \label{eq:rigid-normalization}
  \varphi_r \equiv 0, \quad w_r \equiv 0 \quad \hbox{in }
  \overline{D}.
\end{equation}

By the change of variables $y=Rx$, the pair
$(\varphi_r^+=\varphi_r^+(R^T y), w_r^+=w_r^+(R^Ty))$ satisfies
\begin{equation}
  \label{eq:rigid-27-1}
  \divrg_{y}(S  \nabla_y w_r^+ )= - \divrg_y ( S R \varphi_r^+) \quad \hbox{in }
  C_{2\overline{r}}^{+}
\end{equation}
and
\begin{equation}
  \label{eq:rigid-27-2}
  \divrg_{y}(\widetilde{\mathbb P}(y) \nabla_y \varphi_r^+ )=  S(\varphi_r^+ + R^T \nabla_y w_r^+) \quad \hbox{in }
  C_{2\overline{r}}^{+},
\end{equation}
where $S=S(R^T y)$ and $\widetilde{\mathbb P}(y) [A] = R \mathbb P
(R^{T}y) [R^{T}AR] R^{T}$ for every $2 \times 2$ matrix $A$. The
tensor $ \widetilde{\mathbb P}$ belongs to $C^{0,1}(
C_{2\overline{r}}^{+})$, with $ \| \widetilde{\mathbb P}\|_{
C^{0,1}( C_{2\overline{r}}^{+})  } \leq C h^3$, where $C>0$ only
depends on $\alpha_0$, $\alpha_1$ and $\gamma_0$. Moreover, $
\widetilde{\mathbb P}$ satisfies the strong convexity condition
\eqref{eq:convex-P-Lame}.

Recalling that $w_{r}^{+}=0$ on $\partial D$ and by applying
\eqref{eq:rigid-17-2} with $\rho = \overline{r}$, we have
\begin{equation}
  \label{eq:rigid-27-3}
  \int_{ Q_{j} \cap \partial D } | S \nabla w_r^+ \cdot n  |^{2}
  \leq
  C
  \left ( 1+\frac{\rho_0}{r_D} \right )
  \int_{C_{2\overline{r}}^{+}} \rho_0 |  \nabla w_r^+   |^{2} +
  C \rho_0 \int_{C_{2\overline{r}}^{+}} |\varphi_r^+|^2 + \rho_0^2| \nabla
  \varphi_r^+ |^2 ,
\end{equation}
where $C>0$ only depends on $L_D$, $\alpha_0$, $\gamma_0$,
$\alpha_1$. Similarly, since $\varphi_r^+=0$ on $\partial D$, by
applying estimate \eqref{eq:rigid-18-2} with $\rho = \overline{r}$
we obtain
\begin{equation}
  \label{eq:rigid-27-4}
  \int_{ Q_{j} \cap \partial D } | (\mathbb P \nabla \varphi_r^+) n  |^{2}
  \leq
  C  \left ( 1+\frac{\rho_0}{r_D} \right )
  \int_{C_{2\overline{r}}^{+}}
   \rho_0^3|\varphi_r^+|^2 + \rho_0^5 | \nabla \varphi_r^+ |^2 +
  C \int_{C_{2\overline{r}}^{+}} \rho_0^3 |  \nabla w_r^+   |^{2},
\end{equation}
where $C>0$ only depends on $L_D$, $\alpha_0$, $\gamma_0$,
$\alpha_1$.

Let us consider the cylinder
\begin{equation}
  \label{eq:rigid-28-1}
  C^{*}
  =
  \{(y_1,y_2) \in \R^2 \ | \ |y_1|< \overline{r}, \
    |y_2| < L'\overline{r}
  \},
\end{equation}
where $L'= \max \{ L, 2-L \}$, and let $\varphi_r^* \in H^1( C^*,
\R^2)$ be defined as follows:
\begin{equation}
  \label{eq:rigid-28-2}
  \varphi_r^*=\left\{ \begin{array}{ll}
  \varphi_r^+ & \textrm{in } C_{\overline r}^+
  ,\\
  &  \\
  0 & \textrm{in }  C^{*} \setminus C_{\overline r}^+ .
  \end{array}\right.
\end{equation}
By applying the Poincar\'{e} inequality
$\int_{C_{2\overline{r}}^{+}} |\varphi_r^+|^2 \leq C r_D^2
\int_{C_{2\overline{r}}^{+}} | \nabla \varphi_r^+ |^2$, with $C>0$
only depending on $L_D$, the Korn-type inequality
\eqref{eq:Korn-Kondratev-Oleinik} to $\varphi_r^*$, and by
\eqref{eq:rigid-27-4}, we have
\begin{equation}
  \label{eq:rigid-28-3}
  \int_{ Q_{j} \cap \partial D } | (\mathbb P \nabla \varphi_r^+) n  |^{2}
  \leq
  C  \left ( 1+\frac{\rho_0}{r_D} \right )
  \int_{C_{2\overline{r}}^{+}}
   \rho_0^5 | \widehat{\nabla} \varphi_r^+ |^2 +
  C \int_{C_{2\overline{r}}^{+}} \rho_0^3 |  \nabla w_r^+   |^{2},
\end{equation}
where $C>0$ only depends on $L_D$, $\alpha_0$, $\gamma_0$,
$\alpha_1$.

Finally, in order to estimate locally the $L^2$ norm of the
contact forces $S(\varphi_r^+ + \nabla w_r^+) \cdot n$ on the
boundary of $D$, we rewrite inequality \eqref{eq:rigid-27-3} as
follows
\begin{multline}
  \label{eq:rigid-28-4}
  \int_{ Q_{j} \cap \partial D } | S(\varphi_r^+ + \nabla w_r^+) \cdot n  |^2
  \leq
  \\
  \leq
  C
  \left ( 1+\frac{\rho_0}{r_D} \right )
  \int_{C_{2\overline{r}}^{+}} \rho_0 |  \nabla w_r^+   |^{2} +
  C \rho_0 \int_{C_{2\overline{r}}^{+}}  |\varphi_r^+|^2 + \rho_0^2 | \nabla
  \varphi_r^+ |^2  + C \int_{ Q_{j} \cap \partial D } \rho_0^2 |\varphi_r^+|^2,
\end{multline}
where $C>0$ only depends on $L_D$, $\alpha_0$, $\gamma_0$,
$\alpha_1$. Recalling that $\varphi_r^+ =0$ on $\partial D$, by
using Poincar\'{e} inequalities and the Korn-type inequality
\eqref{eq:Korn-Kondratev-Oleinik}, we have
\begin{multline}
  \label{eq:rigid-29-2}
  \int_{ Q_{j} \cap \partial D } | S(\varphi_r^+ + \nabla w_r^+) \cdot n  |^2
  \leq
    C  \left ( 1+\frac{\rho_0}{r_D} \right )
  \int_{C_{2\overline{r}}^{+}}
    \rho_0 |  \nabla w_r^+   |^{2} +
  C \int_{C_{2\overline{r}}^{+}}  \rho_0^3 | \widehat{\nabla} \varphi_r^+ |^2  ,
\end{multline}
where $C>0$ only depends on $L_D$, $\alpha_0$, $\gamma_0$,
$\alpha_1$.

By summing \eqref{eq:rigid-28-3} and \eqref{eq:rigid-29-2}, using
the normalization \eqref{eq:rigid-normalization}, by applying
Poincar\'{e}'s inequality \eqref{eq:PoincG-a-bis} and the
Korn-type inequality \eqref{eq:Korn-Kondratev-Oleinik} we have
\begin{multline}
  \label{eq:rigid-finale}
  \int_{ Q_{j} \cap \partial D } | (\mathbb P \nabla \varphi_r^+) n  |^{2}
    + \rho_0^2 | S(\varphi_r^+ + \nabla w_r^+) \cdot n  |^2
  \leq
    \\
    \leq
  C  \left ( 1+\frac{\rho_0}{r_D} \right )
  \int_{C_{2\overline{r}}^{+}}
   \rho_0^5 | \widehat{\nabla} \varphi_r^+ |^2 +
  \rho_0^3 |  \nabla w_r^+   |^{2}
    \leq
    \\
    \leq
    C  \left ( 1+\frac{\rho_0}{r_D} \right )
    \left ( \int_{C_{2\overline{r}}^{+}}
   \rho_0^5 | \widehat{\nabla} \varphi_r^+ |^2 +
    \rho_0^3 | \varphi_r^+ + \nabla w_r^+   |^{2}
    + C \int_{C_{2\overline{r}}^{+}} \rho_0^3 | \varphi_r^+ |^{2}
    \right )
    \leq
    \\
    \leq
    C  \left ( 1+\frac{\rho_0}{r_D} \right )
  \int_{C_{2\overline{r}}^{+}}
   \rho_0^5 | \widehat{\nabla} \varphi_r^+ |^2 +
    \rho_0^3 | \varphi_r^+ + \nabla w_r^+   |^{2},
\end{multline}
where $C>0$ only depends on $L_D$, $\alpha_0$, $\gamma_0$,
$\alpha_1$.

Since $ 1 + \rho_0/r_D \leq (1+d_0/2)\rho_0/r_D$, and recalling
\eqref{eq:rigid-25-2}, we obtain the wished estimate
\eqref{eq:rigid-L2-estimate-partialD}.
\end{proof}

We conclude the section with a proof of Proposition
\ref{prop:Local-Boundary-Estimates-L2}, which is based on the
following result.

\begin{lem}
\label{Rellich}
Let $S$, $\mathbb P$ given in \eqref{eq:shearing-tensor},
\eqref{eq:bending-tensor} satisfy the strong convexity conditions
\eqref{eq:convex-S-Lame}, \eqref{eq:convex-P-Lame} and the regularity conditions in \eqref{eq:lip_operatori}.

For every $w \in H^{3/2}(C_{2 \rho}^{+})$ such that ${\rm div}(S
\nabla w ) \in L^{2}(C_{2 \rho}^{+})$ and $w=| \nabla w |=0$ on
  $\partial C_{2 \rho}^{+} \setminus \Delta_{2 \rho}$, we have
\begin{equation}
  \label{eq:rigid-16-1}
  \int_{\Delta_{\rho}} | S \nabla w \cdot n|^{2}
  \leq
  C
  \left (
  \rho_0^2 \int_{\Delta_{2 \rho}}  \rho_0 | \nabla_{T} w|^{2}
  +
  \rho_0 \int_{C_{2 \rho}^{+}} | \nabla w|^{2}
  +
  | \nabla w||{\rm div}(S \nabla w)|
  \right
  ),
\end{equation}
where $C>0$ only depends on $L$, $\alpha_0$, $\gamma_0$,
$\alpha_1$.

For every $\varphi \in H^{3/2}(C_{2 \rho}^{+}, \R^2)$ such that
${\rm div}(\mathbb P \nabla \varphi ) \in L^{2}(C_{2 \rho}^{+},
\R^2)$ and $|\varphi|=| \nabla \varphi |=0$ on
  $\partial C_{2 \rho}^{+} \setminus \Delta_{2 \rho}$, we have
\begin{equation}
  \label{eq:rigid-16-2}
  \int_{\Delta_{\rho}} | (\mathbb P \nabla \varphi) n|^{2}
  \leq
  C
  \left (
  \int_{\Delta_{2 \rho}} \rho_0^6 | \nabla_{T} \varphi |^{2}
  +
  \int_{C_{2 \rho}^{+}} \rho_0^5 | \nabla \varphi |^{2}
  +
  \rho_0^3 | \nabla \varphi ||{\rm div}(\mathbb P \nabla \varphi)|
  \right
  ),
\end{equation}
where $C>0$ only depends on $L$, $\alpha_0$, $\gamma_0$,
$\alpha_1$.
\end{lem}
\begin{proof}
The proof follows the lines of the proof of the analogous result
obtained in conductivity and elasticity context, see
\cite{l:amr02} (Lemma $5.2$) and \cite{l:mr03} (Lemma $4.3$),
respectively. The key mathematical tool is a generalization of the
well-known Rellich's identity \cite{l:r}.
\end{proof}

\begin{proof}[Proof of Proposition
\ref{prop:Local-Boundary-Estimates-L2}]
The proof can be obtained by adapting the arguments used, for
example, in the proof of the analogous result in three-dimensional
elasticity \cite{l:mr03} (Proposition $4.2$), see also
\cite{l:amr02} (Proposition $5.1$). Moreover, the proof of the
estimates \eqref{eq:rigid-17-2} and \eqref{eq:rigid-18-2} follows
the same path. Therefore, we sketch the proof of the inequality
\eqref{eq:rigid-17-2} only.

We first prove the thesis under the additional assumption that $w
\in H^{3/2}(C_{2 t}^{+})$ for every $t< \rho$.

Let us introduce the cut-off function in $\R^{2}$
\begin{equation}
  \label{eq:rigid-18-3}
  \eta(x_1,x_2)=\chi(x_1)\Psi(x_2),
\end{equation}
where
\begin{equation}
  \label{eq:rigid-18-3a}
  \chi \in C_{0}^{\infty} ( \R), \quad \chi(x_1)\equiv 1
  \ \ \textrm{if} \  |x_1|\leq\rho,
\end{equation}
\begin{equation}
  \label{eq:rigid-18-3b}
    \chi(x_1)\equiv 0 \ \
  \textrm{if} \  |x_1|\geq\frac{3}{2}\rho,
\end{equation}
\begin{equation}
  \label{eq:rigid-18-3c}
  \|  \chi' \|_{\infty} \leq C_{1}
  \rho^{-1}, \quad
  \|  \chi'' \|_{\infty} \leq C_{1}
  \rho^{-2},
\end{equation}
\begin{equation}
  \label{eq:rigid-18-3d}
  \Psi \in C_{0}^{\infty} ( \R), \quad \psi( x_{2}) \equiv 1
  \ \ \textrm{if} \  |x _{2} |\leq\rho L,
\end{equation}
\begin{equation}
  \label{eq:rigid-18-3e}
    \psi(x_{2} )\equiv 0 \ \
  \textrm{if} \  |x _2|\geq\frac{3}{2}\rho L,
\end{equation}
\begin{equation}
  \label{eq:rigid-18-3f}
  \|  \psi' \|_{\infty} \leq C_{2}
  \rho^{-1}, \quad
  \|  \psi'' \|_{\infty} \leq C_{2}
  \rho^{-2},
\end{equation}
where $C_{1}$ is an absolute constant and $C_{2}$ is a constant
only depending on $L$.

For every $c \in \R$ the function
\begin{equation}
  \label{eq:rigid-19-1}
  u=\eta(w-c)
\end{equation}
satisfies the hypotheses of Lemma \ref{Rellich} with $\rho=t$, for
every $t \in \left ( \frac{3}{4} \rho, \rho \right)$.

By substituting \eqref{eq:rigid-19-1} in \eqref{eq:rigid-16-1},
and recalling \eqref{eq:rigid-17-1}, we have
\begin{multline}
  \label{eq:rigid-21-1}
    \int_{\Delta_t} S^2
    \left (
    (w-c)^2 |\nabla \eta \cdot n|^2 + \eta^2 |\nabla w \cdot n|^2
    +2\eta(w-c)(\nabla \eta \cdot n) (\nabla w \cdot n)
    \right )
    \leq
    \\
    \leq
    C \rho_0^2
    \int_{\Delta_{2t}} (w-c)^2 |\nabla_T \eta|^2 + \eta^2
    |\nabla_T w|^2 + 2\eta(w-c) \nabla_T \eta \cdot \nabla_T w
    +
    \\
    +
    C \rho_0
    \int_{C_{2t}^+} (w-c)^2 |\nabla \eta|^2 + \eta^2 |\nabla w|^2
    + 2 \eta(w-c)\nabla \eta \cdot \nabla w + |\varphi|^2 +
    \rho_0^2 |\nabla \varphi|^2,
\end{multline}
where $C>0$ only depends on $L$, $\alpha_0$, $\gamma_0$,
$\alpha_1$.

By recalling \eqref{eq:rigid-18-3a}--\eqref{eq:rigid-18-3f}, by
using Schwarz inequality and $2ab \leq a^2/\epsilon + \epsilon
b^2$, for every $\epsilon >0$, we obtain
\begin{multline}
  \label{eq:rigid-22-1}
  \int_{\Delta_{t}} |  S \nabla w \cdot n   |^{2}
  \leq
    \\
    \leq
  C
  \left (
  \rho_0^2
    \int_{\Delta_{2t}}  \frac{(w-c)^2}{t^2}  +  | \nabla_{T} w|^{2}
    +
  \rho_0
    \int_{C_{2t}^{+}} \frac{(w-c)^2 }{ t^2 } + |\nabla w |^2 +
  |\varphi|^2 + \rho_0^2 |\nabla \varphi|^2
  \right ),
  \\
  \quad \textrm{for every }  t \in \left ( \frac{3}{4} \rho, \rho \right ),
\end{multline}
where $C>0$ only depends on $L$, $\alpha_0$, $\gamma_0$,
$\alpha_1$.

Choosing $c=\frac{1}{|C_{2t}^{+}| } \int_{C_{2t}^{+} } w$, by
applying trace inequalities and Poincar\'{e}'s inequality
\eqref{eq:PoincG-a}, we have
\begin{multline}
  \label{eq:rigid-23-3}
  \int_{\Delta_{t}} |  S \nabla w \cdot n   |^{2}
  \leq
    \\
    \leq
  C
  \left (
  \rho_0^2 \int_{\Delta_{2t}}  | \nabla_{T} w|^{2} +\left ( 1 +
  \frac{\rho_0}{t} \right) \int_{C_{2t}^+} \rho_0 | \nabla w|^{2}
  +
  \int_{C_{2t}^{+}} \rho_0 |\varphi|^2 + \rho_0^3 |\nabla \varphi|^2
  \right ),
  \\
  \quad \textrm{for every }  t \in \left ( \frac{3}{4} \rho, \rho \right ),
\end{multline}
where $C>0$ only depends on $L$, $\alpha_0$, $\gamma_0$,
$\alpha_1$. Passing to the limit for $t \rightarrow \rho$, we
obtain \eqref{eq:rigid-17-2}.

We notice that if the function $\Psi$ representing the boundary
$\Delta_{2t}$ is smooth, then the additional assumption made at
the beginning of this proof (e.g., $w \in H^{3/2}(C_{2 t}^{+})$
for every $t< \rho$) is satisfied by regularity estimates up to
the boundary for solutions to \eqref{eq:rigid-17-1}. When
$\Delta_{2t}$ is represented by a Lipschitz function, the thesis
can be obtained by following the approximation argument presented
in \cite{l:mr03} (Step $2$ of Proposition $4.2$).
\end{proof}

\end{document}